\documentclass[11pt]{article}
\usepackage{graphicx} 
\usepackage{lipsum}	
\usepackage{url}              
\usepackage[super]{nth}
\usepackage{longtable}

\usepackage{chemformula}
\usepackage{color} 
\usepackage[version=4]{mhchem}
\usepackage{siunitx}
\usepackage{chemfig}
\usepackage{subcaption}           
\usepackage{amssymb, amsmath}
\usepackage{mathtools}
\usepackage{enumerate} 
\usepackage{amsthm}
\usepackage{mathrsfs}
\usepackage{verbatim}
\newcommand{\iR}{\mathbb{R}}
\newcommand{\iN}{\mathbb{N}}
\newcommand{\iZ}{\mathbb{Z}}

\newcommand{\po}{\mathscr{P}_{1}}
\newcommand{\poo}{\mathscr{P}_{1,1}}
\newcommand{\pn}{\mathscr{P}^N}

\usepackage{bbm}

\usepackage{extarrows}
\usepackage{dsfont}
\usepackage{tikz}

\newcommand{\R}{\mathbb{R}}
\newcommand{\Z}{\mathbb{Z}}
\newcommand{\N}{\mathbb{N}}

\newcommand{\lno}{\left\lVert}
\newcommand{\rno}{\right\rVert}

\newcommand{\lk}{\left (}
\newcommand{\rk}{\right )}
\newcommand{\lkg}{\left \{}
\newcommand{\rkg}{\right \}}
\newcommand{\lv}{\left\lvert}
\newcommand{\rv}{\right\rvert}
\newcommand{\lke}{\left [}
\newcommand{\rke}{\right ]}
\newcommand{\loz}{{\ell_1(\Z)}}
\newcommand{\lozp}{{\ell_1^+(\Z)}}
\newcommand{\wloz}{{\ell_{1,1}(\Z)}}
\newcommand{\wlozp}{{\ell_{1,1}^+(\Z)}}
\newcommand{\rnoo}{\rno_1}
\newcommand{\rnooo}{\rno_{1,1}}

\newcommand{\skz}{\sum_{k\in\Z}}

\newtheorem{theorem}{Theorem}[section]
\newtheorem{proposition}[theorem]{Proposition}
\newtheorem{lemma}[theorem]{Lemma}
\newtheorem{cor}[theorem]{Corollary}
\theoremstyle{definition}
\newtheorem{definition}[theorem]{Definition}
\newtheorem{exa}[theorem]{Example}
\newtheorem{re}[theorem]{Remark}

\usepackage{hyperref}
\hypersetup{
pdftitle={Particle Dynamics Driven by Charge Exchange},
pdfauthor={Adrian Schmautz, Rico Zacher},
pdfkeywords={particle dynamics, charge exchange, detailed balance, equilibrium, relative entropy, stability},
colorlinks=false,
pdfborder=0 0 0	
}

\title{Particle dynamics driven by charge exchange}
\author{
  Adrian Schmautz\thanks{Institut für Angewandte Analysis, Universität Ulm, Helmholtzstraße 18, 89081 Ulm, Germany. \\E-mail: \texttt{adrian.schmautz@uni-ulm.de}}
  \and
  Rico Zacher\thanks{Institut für Angewandte Analysis, Universität Ulm, Helmholtzstraße 18, 89081 Ulm, Germany. \\E-mail: \texttt{rico.zacher@uni-ulm.de}}
}
\date{}
\begin{document}

\maketitle
\abstract{We introduce and analyse a mathematical model describing the dynamics of particles generated by charge-exchange interactions. The model extends the well-established exchange-driven growth model, previously studied in \cite{BNKr,EiSch,Es,EsVel,EDG}, by allowing for particle densities defined on the entire integer lattice.
Despite the many similarities between the two models, substantial differences arise both in their qualitative behaviour and in their mathematical analysis. Under suitable assumptions on the kernel in the collision operator, we establish global well-posedness in the class of nonnegative densities with finite first moment. Moreover, under a detailed balance condition, we investigate the structure of equilibria and analyse their stability by means of entropy methods.}

\medskip
\noindent {\bf AMS subjection classification:} 34A34, 37B25, 47J35, 82C05


\section{Introduction}
In this paper we introduce and analyse a model which describes the dynamics of particles arising from charge exchange between particles. We assume that each particle has an integer charge and that two particles can interact with each other by one particle transferring a unit of charge to the other.
Denoting by $X_k$ a particle with charge $k$ ($k\in\Z)$, the interactions can be interpreted as chemical reactions 
\begin{align} \label{chemreact}
X_k+X_{l-1}\ce{ <=>[\text{$K(k,l-1)$}][\text{$K(l,k-1)$}]}X_l+X_{k-1}, \quad k,l\in\Z, 
\end{align}
where $K(k,l)$ is the rate of exchange of a single unit of charge from a particle of charge $k$ to a particle of charge $l$ ($k,l\in\Z$).
Denoting the density of particles with charge $k\in\Z$ at time $t\geq0$ by $f(t,k)$ and assuming mass-action kinetics for \eqref{chemreact} the particle system is described by the equation
\begin{gather}
\begin{aligned}\label{ODE1}
    \partial_t f(t,k)=&\sum_{l\in\Z}\Big(K(l,k-1)f(t,l)f(t,k-1)-K(k,l-1)f(t,k)f(t,l-1)\\&\quad-K(l,k)f(t,l)f(t,k)+K(k+1,l-1)f(t,k+1)f(t,l-1)\Big),
    \end{aligned}
\end{gather}
for $t>0$ and $k\in \iZ$, together with an initial condition
\begin{equation} \label{ODE1IC}
    f(0,k)=f_0(k),\quad k\in \iZ.
\end{equation}

The first and fourth term inside the bracket in \eqref{ODE1} represent the formation of particles with charge $k$ (either a particle with charge $k-1$ acquires a unit of charge or a particle with charge $k+1$ releases a unit of charge). The second and third term describe the loss of particles with charge $k$ (a particle with charge $k$ transfers/imports a unit of charge to/from another particle).

Our model can be viewed as an extension of the exchange-driven growth (EDG) model studied, among others, in \cite{BNKr,EiSch,Es,EsVel,EDG,SiGi}, which describes the dynamics of cluster growth through the exchange of single units called monomers, see also \cite{BdCPS} for a more general exchange-driven system.
In the EDG model, $f(t,k)$ represents the density of clusters of size $k\in \iN_0$ at time $t$. The difference to our model consists in that only nonnegative integers $k$ are considered. As noted in \cite{Es,EDG}, the EDG model itself is a generalisation of a model that is very close to the celebrated Becker-Döring (BD) equations, which were first proposed in \cite{BD} and later modified in \cite{PL}. This Becker-Döring type model corresponds to the chemical reaction network 
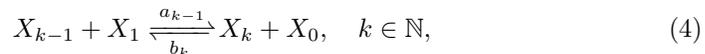
\begin{align} \label{chemreact2}
X_{k-1}+X_{1}\ce{ <=>[\text{$a_{k-1}$}][\text{$b_k$}]}X_k+X_{0}, \quad k\in \iN,
\end{align}
while the BD model has no $X_0$ species and considers the reactions
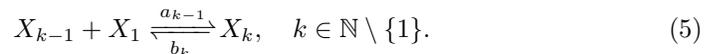
\begin{align} \label{chemreact3}
X_{k-1}+X_{1}\ce{ <=>[\text{$a_{k-1}$}][\text{$b_k$}]}X_k, \quad k\in \iN\setminus\{1\}.
\end{align}
For basic properties and results on the longtime behaviour of the  
Becker-D\"oring equations we refer to \cite{Ball88,Ball,CEL,Nh02,Sl1,Sl2}.

The charge exchange (CE) model \eqref{ODE1} possesses two conserved quantities:
\begin{itemize}
    \item the total mass $\sum_{k\in \iZ}f(t,k)$ (i.e.\ the total number of particles) and
    \item the total charge $\sum_{k\in \iZ}k f(t,k)$, 
\end{itemize}   
see Proposition~\ref{preservationprop}. In order to ensure that both quantities are well-defined for $f(t,\cdot)$ it is natural to choose the base space 
\begin{align*}
    \ell_{1,1}(\Z)\vcentcolon=\Big\{ g\in\loz:\lno g\rnooo\vcentcolon=\skz (1+\lv k\rv)\lv g(k)\rv<\infty\Big\}.
\end{align*}
For $g\in \loz$ we will denote the total mass by $m(g)$, that is,
$m(g)=\sum_{k\in \iZ}g(k)$. The total charge of $g\in  \ell_{1,1}(\Z)$ will be denoted by $q(g)$, that is, $q(g)=\sum_{k\in \iZ}k g(k)$.
Let us further introduce the absolute total charge $|q|(g)$ by
$|q|(g)\vcentcolon=\sum_{k\in \iZ}|k| g(k)$.

Instead of the charge of a particle, $k\in \iZ$ can also be thought of as another quantity that can take on positive and negative multiples of a certain elementary value such as money. In this sense, the CE model - just like the EDG model (see \cite{IKR}) - can be used to describe wealth exchange. Another interesting interpretation is to think of $(X_k,X_{-k})$ ($k\in \iN)$ as pairs of particles and antiparticles of a certain type $k$. 

Even though the CE model appears to be merely a simple extension of EDG from $\iN$ to $\iZ$, there are significant differences that make the analysis of EC considerably more difficult. For example, in the EDG case, the two conserved quantities 
$\sum_{k=0}^\infty f(t,k)$ (total number of particles) and 
$\sum_{k=0}^\infty k f(t,k)$ (total mass) with nonnegative $f$
immediately yield an a priori bound in the natural base space 
$\ell_{1,1}(\iN_0)$. This is no longer the case for CE as we obviously do not
get a bound for $\sum_{k\in \iZ}|k| f(t,k)$. 

At the particle level, this difference between the two models can be clearly illustrated as follows. Suppose we only have two particles and regard $k$ as position (i.e., as spatial variable). 
In the EDG case, it is not possible that one of the particles escapes to $k=\infty$, since if one particle moves a unit to the right, the other one has to make a step to the left. Due to the ``barrier'' at $k=0$, this process of moving away from each other has to stop after finitely many steps. In contrast, in the CE case, both particles can move to infinity, one to $+\infty$ and one to $-\infty$.

To illustrate this point even further, let us consider the CE model in the case of a constant kernel $K$, say $K=1$. Assuming that the total mass is equal to one, that is, $f(t,\cdot)$ is a probability density for all $t\ge 0$, equation \eqref{ODE1} amounts to the linear discrete diffusion equation on $\iZ$
\begin{equation} \label{DiscreteHeat}
\partial_t f(t,k)=f(t,k+1)-2f(t,k)+f(t,k-1).
\end{equation}
A brief formal calculation shows that
\begin{equation} \label{DerivAbsCharge}
\frac{d}{dt}|q|(f(t,\cdot))=\frac{d}{dt}\sum_{k\in \iZ}|k| f(t,k)=2f(t,0)\ge 0,
\end{equation}
that is, the absolute total charge is increasing. By means of the fundamental solution of \eqref{DiscreteHeat} one can also verify that $|q|(f(t,\cdot)\to \infty$ as $t\to \infty$, see the appendix. Under the same assumptions, in the EDG case, the situation is completely different. Here, for every initial density $f_0$ in $\ell_{1,1}(\iN_0)$, the corresponding solution converges
in $\ell_{1,1}(\iN_0)$ to an exponentially decaying equilibrium as $t\to \infty$, cf.\ \cite{EDG}.

We will now describe our main results. Throughout most of the article, we will assume that the kernel $K$ satisfies the following condition.
\begin{itemize}
\item [{\bf ($\mathcal{B}$)}] $K:\Z^2\rightarrow[0,\infty)$ and $\sup_{k,l\in\Z}K(k,l)\leq C_K$ for some $C_K\in (0,\infty)$.
\end{itemize}
Under this assumption, it is not difficult to show the local well-posedness of the initial-value problem \eqref{ODE1},\eqref{ODE1IC} in the spaces $\loz$ and $\ell_{1,1}(\Z)$, see Theorem \ref{LocalWP} below. Since we are mainly interested in densities we will then restrict ourselves to the positive cone of $\loz$ and $\wloz$, respectively, given by
\[
\lozp\vcentcolon=\lkg g\in\loz:g(k)\geq0\text{ for all }k\in\Z\rkg \text{ and}\;\;
    \wlozp\vcentcolon= \wloz\cap \lozp.
\]
For nonnegative initial values $f_0\in X$ with $X\in \{\loz,\wloz\}$, we establish global existence of nonnegative solutions in $C^1([0,\infty);X)$ and show that total mass 
and, if $X=\wloz$, also total charge are preserved, see Corollary  \ref{l1globex} and Theorem \ref{l11globex}. Positivity is shown by means of an abstract quasi-positivity result due to Volkmann \cite{Volkmann}.
Assuming in addition that the kernel is strictly positive in the sense that
\begin{itemize}
\item [{\bf ($\mathcal{P}$)}] $K(k,l)>0$ for all $k,l\in \iZ$,
\end{itemize}
then any solution starting from a nontrivial nonnegative initial value becomes instantaneously strictly positive on $\iZ$, cf.\ Proposition \ref{propfg0}.

For several technical reasons, it is very useful to have results which allow to approximate the solution of \eqref{ODE1},\eqref{ODE1IC} by a solution of a related finite-dimensional problem. Similarly as in the EDG case \cite{EDG},
we are able to prove that on any compact time interval $[0,T]$, any nonnegative $\wloz$-solution can be approximated arbitrarily well by the solution of a corresponding truncated system which lives on $S_N\vcentcolon=\{k\in\Z:|k|\leq N\}$ with sufficiently large $N$, see Theorem \ref{truncapproxc10T}.

Concerning long-time behaviour, in addition to ($\mathcal{P}$), we impose the crucial assumption that the kernel $K$ admits detailed balancing, i.e.\ we assume the following. 
\begin{itemize}
\item [{\bf ($\mathcal{DB}$)}] There exists $g:\,\iZ\to (0,\infty)$ such that  
\[
K(k,l-1)g(k)g(l-1)=K(l,k-1)g(l)g(k-1)\quad \text{for all}\; k,l\in \iZ.
\]
\end{itemize}
The detailed balance condition plays a key role in kinetic theory. In the context of chemical reaction systems, it states that at equilibrium, each reaction is exactly balanced by its reverse. Detailed balance provides a strong structural property that enables the use of entropy methods to prove stability and convergence to equilibrium. This condition has also been used in previous work on the BD equations and the EDG model.

The function $g$ in ($\mathcal{DB}$) is formally an equilibrium of \eqref{ODE1}, but like other existing equilibria, it does not necessarily belong to $\wloz$.  
In order to ensure existence of (detailed balanced) equilibria
in $\wloz$, we require additional assumptions on the kernel $K$. We assume the following here. 
\begin{itemize}
\item [{\bf ($\mathcal{E}$)}]
The limits    
    \begin{align*}
        \Lambda_+\vcentcolon=\lim_{j\rightarrow\infty}\frac{K(j,0)}{K(1,j-1)}\in (0,\infty],\quad 
        \Lambda_-\vcentcolon=\lim_{j\rightarrow-\infty}\frac{K(0,j)}{K(j+1,-1)}\in (0,\infty]
    \end{align*}
    exist and $ \mathbf{I}_K\vcentcolon=(\phi_-,\phi_+)\vcentcolon=\lk\lk \Lambda_-\kappa\rk^{-1},\Lambda_+\kappa\rk\neq\emptyset$, where 
    $\kappa\vcentcolon=\sqrt{\frac{K(1,-1)}{K(0,0)}}$.
\end{itemize}

Note that, by scaling, we may restrict ourselves to (nonnegative) solutions with total mass equal to one. The assumptions ($\mathcal{P}$), ($\mathcal{DB}$) and ($\mathcal{E}$) lead to a one-parameter family $f_\phi$ of (strictly positive) detailed balanced equilibria in $\wloz$ with total mass equal to one and parameter 
$\phi\in  \mathbf{I}_K$. Moreover, the total charge $q(f_\phi)$, considered as a function of $\phi$, is $C^1$ and strictly increasing in
$ \mathbf{I}_K$, see Lemma \ref{ddtqequi}. Given an $f_0\in \wlozp$ with total mass one and total charge $q$ such that
\begin{equation} \label{subcrit}
q\in \mathbf{J}_K\vcentcolon=\big(\lim_{\phi\to \phi_-}q(f_\phi),\lim_{\phi\to \phi_+}q(f_\phi)\big),
\end{equation}
there exists a unique $\phi\in  \mathbf{I}_K$ for which $f_\phi$ has the same total charge as $f_0$. This equilibrium is the natural candidate for the limit as $t\to \infty$ of the solution $f$ starting at $f_0$. In fact, one can show that, in this situation, there are no further equilibria with total mass one and total charge equal to $q(f_0)$. Proving that the solution indeed converges as $t\to \infty$
is a challenging task and seems to require further assumptions on the kernel. For the EDG model, such a convergence result could be established in \cite{EDG} under several additional assumptions on the kernel. In the present paper, we prove several basic properties with regard to long-time behaviour, in particular we show stability of the family
of equilibria $f_\phi$ for $\phi\in  \mathbf{I}_K$, see Theorem \ref{StabTheorem} below. Convergence to equilibrium will be studied in a forthcoming paper.

An important tool for studying the long-time behaviour is relative entropy. As already mentioned, the detailed balance condition ($\mathcal{DB}$) plays a central role in this context. Analogously to the EDG model and the Becker-D\"{o}ring system, we can show that the relative entropy of nonnegative $\wloz$-solutions with respect to any (strictly positive) detailed balance equilibrium of unit mass is nonincreasing. Applying this property to an exponentially decaying equilibrium yields an {\em a priori} bound for nonnegative $\wloz$-solutions in $\wloz$. This estimate constitutes a crucial step towards establishing compactness properties of the positive semiorbit associated with nonnegative solutions. We note that, for the BD equations and the EDG model, the corresponding estimate in ${\ell_{1,1}(\N)}$ and ${\ell_{1,1}(\N_0)}$, respectively, follows immediately from the conservation of total mass. 
Relative entropy, combined with suitable weighted Csisz\'{a}r-Kullback-Pinsker inequalities from \cite{BV}, also serves as the key ingredient in the proof of stability of the equilibria $f_\phi$ for $\phi\in  \mathbf{I}_K$ (Theorem \ref{StabTheorem}).

The situation in which \eqref{subcrit} holds is referred to as the {\em subcritical} case. In this regime, there exists
exactly one equilibrium in $\wloz$ with the same total mass and total charge. If $q>\lim_{\phi\to \phi_+}q(f_\phi)$ 
or $q<\lim_{\phi\to \phi_-}q(f_\phi)$ (which in particular means that at least one of the two limits is finite), then the system is in the {\em supercritical} regime. This case is more delicate analytically. In analogy with the Becker–D\"{o}ring equations and the EDG model -- where, in the supercritical regime, mass can be lost in the limit $t\to \infty$ -- we expect that some amount of charge can be likewise lost or gained at $t\to \infty$ and that supercritical solutions converge to a critical equilibrium in a weak sense. This is an interesting open problem and the subject of further investigation.

One of the aims of this work is also to demonstrate that certain results can be achieved through alternative methods.
For example, local well-posedness and preservation of positivity are directly proven for the infinite-dimensional problem \eqref{ODE1},\eqref{ODE1IC}, without passing through
a finite-dimensional truncation. Our arguments rely instead
on abstract functional analytic results. 

The paper is organised as follows. Section 2 is devoted to well-posedness, positivity, conserved quantities and scaling properties. In Section 3 we derive lower bounds for solutions and establish strict positivity. Section 4 addresses approximation by truncation. Section 5 discusses the detailed balance condition and derives the family $f_\phi$; in addition, we provide a characterisation of ($\mathcal{DB}$) in terms of an extended Becker-Döring condition. In Section 6, we establish various fundamental properties of the relative entropy with respect to detailed balance equilibria, in particular showing that it serves as a Lyapunov function. Section 7 is devoted to the stability of the equilibria $f_\phi$. Finally, the appendix collects
various auxiliary results.
%**********************************************************
\section{Well-posedness}
%************************************************************
Throughout this section we always assume {\bf($\mathcal{B}$)}.

For $f:\,\iZ\to \iR$ and $g\in\loz$, we define the function $Q(f,g):\,\iZ\to \iR$ by
\begin{equation} \label{QDEF}
\begin{split}
     Q(f,g)(k)&\vcentcolon=\sum_{l\in\Z}\Big(K(l,k-1)g(l)f(k-1)-K(k,l-1)f(k)g(l-1)\\&\quad\quad-K(l,k)g(l)f(k)+K(k+1,l-1)f(k+1)g(l-1)\Big),\;k\in \iZ.
\end{split}
\end{equation}
We also set $Q(f)\vcentcolon=Q(f,f)$ for $f\in\loz$. Observe that $Q(f,g)$
is well defined, thanks to condition ($\mathcal{B}$); see also Proposition \ref{basicpropQ}. Equation \eqref{ODE1} can then be rewritten as 
\begin{align}\label{ODE}
    \partial_t f(t,k)&=Q(f(t,\cdot))(k),\quad t>0,\,k\in \iZ.
\end{align}

The following proposition collects basic mapping properties of the operator $Q$.
%************************************************************
\begin{proposition} \label{basicpropQ}
(i) For $f,g\in\loz$ and $c:\,\iZ\to [0,\infty)$
    \begin{align*}
        \sum_{k\in\Z}c(k)\lv Q(f,g)(k)\rv\leq C_K\lno g\rno_1\sum_{k\in\Z}|f(k)|\lke c(k+1)+2c(k)+c(k-1)\rke.
    \end{align*}

\noindent (ii) Let $X\in \{\loz,\wloz\}$. Then for $f\in X$ and $g\in\loz$
    \begin{align} \label{basicprop2}
        \lno Q(f,g)\rno_X&\leq C_XC_K\lno f\rno_X\lno g\rno_1,
    \end{align}
    where $C_{X}=4$ if $X=\ell_1(\Z)$, and $C_{X}=6$ if $X=\wloz$.    

\medskip
\noindent (iii) Let $X\in \{\loz,\wloz\}$. For $f,g\in X$ there holds 
    \begin{align*}
        \lno Q(f)-Q(g)\rno_X&\leq C_{X}C_K\lk\lno f\rno_X+\lno g\rno_X\rk\lno f-g\rno_X.
    \end{align*}     
In particular, $Q:\,X\to X$ is Lipschitz continuous on bounded sets: for every $R>0$ there exists $L(R)>0$ such that for all
$f,g\in X$ with $\lno f \rno_X\le R$, $\lno g \rno_X\le R$ we have
\begin{align*}
        \lno Q(f)-Q(g)\rno_X&\leq L(R)\lno f-g\rno_X,
    \end{align*}
    and the map $R\mapsto L(R)$ is nondecreasing.
\end{proposition}
%***********************************************************
\begin{proof} (i) Let $f,g\in\loz$ and $c:\,\iZ\to [0,\infty)$. Then 
    \begin{align*}
    \sum_{k\in\Z}c(k)\lv Q(f,g)(k)\rv&\leq \sum_{k,l\in\Z}c(k)\lv g(l)\rv\big( K(l,k-1)|f(k-1)|+K(k,l)|f(k)|\\
    &\quad\quad\quad\quad\quad\quad+K(l,k)|f(k)|+K(k+1,l)|f(k+1)|\big)\\
    &\le C_K\lno g\rno_1\sum_{k\in\Z}c_k\big( |f(k+1)|+2|f(k)|+|f(k-1)|\big)\\
    &=C_K\lno g\rno_1\sum_{k\in\Z}|f(k)|\big( c(k+1)+2c(k)+c(k-1)\big),
\end{align*}
where we may rearrange the sums as every summand is nonnegative.

(ii)  We use (i). For $X=\loz$ we choose $c(k)=1$, $k\in\Z$, and obtain \eqref{basicprop2} with $C_\loz=4$. Moreover, the choice $c(k)=1+|k|$ results in \eqref{basicprop2} for $X=\wloz$ with $C_\wloz=6$, as 
\begin{align*}
    c(k+1)+c(k-1)&=1+|k+1|+1+|k-1|\\
    & \leq4+2|k|\leq 4(1+|k|)=4c(k).
\end{align*}

(iii) This follows from (ii), the bilinearity of $Q$ and the
inequality $\lno f\rno_1\le \lno f\rno_{1,1}$ for $f\in\wloz$.
\end{proof}

In what follows we will view \eqref{ODE} as an evolution equation in a Banach space $X$, where we are mainly interested in $X=\loz$ and $X=\wloz$. That is, we will study the initial value problem
\begin{equation}\label{P}
    \begin{cases}
        \dot{f}(t)&=\,Q(f(t)),\quad t>0,\\
        f(0)\!\!&=\,f_0,
    \end{cases}\tag{P}
\end{equation}
where $f_0\in X$ and $f$ depends on time and takes values in $X$.  

We have the following local well-posedness result for general initial data in $\loz$ and $\wloz$, respectively. 
%**********************************************************
\begin{theorem} \label{LocalWP}
Let $X\in \{\loz,\wloz\}$ and $f_0\in X$. Then the following assertions hold true.
\begin{itemize}
    \item [(i)] There exists a unique maximal ($X$-) solution $f=\varphi(\,\cdot\,;f_0)\in C^1([0,t^+(f_0));X)$ of \eqref{P}.
    \item [(ii)] If $t^+(f_0)<\infty$, then $\lim_{t\to t^+(f_0)}\lno f(t)\rno_X=\infty$.
    \item [(iii)] For every $T\in (0,t^+(f_0))$, there exists $\delta=\delta(f_0,T)>0$ such that $t^+(g_0)>T$ for all 
    $g_0\in \bar{B}_X(f_0,\delta)$. Moreover, the map
    \[
    \bar{B}_X(f_0,\delta)\to C^1([0,T];X),\quad g_0\mapsto \varphi(\,\cdot\,;g_0),
    \]
    is Lipschitz continuous.
\end{itemize}
\end{theorem}
%************************************************************
\begin{proof}
The statements follow from the general well-posedness theory for abstract semilinear evolution equations with 
nonlinearities that are Lipschitz continuous on bounded sets, which in our case applies with the trivial semigroup generator $A=0$. We refer to \cite[Section 7.1]{Lun} and \cite[Theorem 3.4]{Schnau}.
\end{proof}
%**********************************************************

We next study various preservation properties of \eqref{ODE}. We say that $f\in\loz$ is \emph{symmetric}, if $f(k)=f(-k)$ for all $k\in\N$. Furthermore, \eqref{ODE} is said to \emph{preserve symmetry} if for every symmetric initial value $f_0\in \loz$, the corresponding solution $f(t)=\varphi(t;f_0)\in \loz$ (provided by Theorem \ref{LocalWP}) is symmetric for all $t\in [0,t^+(f_0))$.
%********************************************************
\begin{proposition} \label{partInt}
    Let $f,g\in \loz$ and $c:\,\iZ\to \iR$. Assume that
\begin{equation} \label{absconvass}
\sum_{k\in\Z}|c(k)f(k+\sigma)|<\infty\quad \mbox{for}\; \sigma\in\{-1,0,1\}.
\end{equation} 
Then there holds
    \begin{align*}
        \sum_{k\in\Z}c(k) Q(f,g)(k)=\sum_{k,l\in\Z}f(k)g(l)&\Big( K(l,k)\big[ c(k+1)-c(k)\big]\\&\quad+K(k,l)\big[ c(k-1)-c(k)\big]\Big).
    \end{align*}
\end{proposition}
%********************************************************
\begin{proof}
Inserting the definition of $Q$ in $\sum_{k\in\Z}c(k) Q(f,g)(k)$, assumption \eqref{absconvass}, $g\in \loz$, and ($\mathcal{B}$) ensure the absolute convergence of each of the four double sums which arise when applying the triangle inequality to the big bracket in \eqref{QDEF}. This allows us to rearrange the terms appropriately, which, combined with a few index shifts, easily leads to the assertion.   
\end{proof} 
%*************************************************************
%************************************************************
\begin{proposition}\label{preservationprop}
Let $X\in \{\loz,\wloz\}$, $f_0\in X$, and let $f$ be the corresponding $X$-solution of \eqref{P} on $J\vcentcolon=[0,t^+(f_0))$. Set $X^+=\{g\in X:\,g(k)\ge 0,\,k\in \Z\}$. Then the following statements hold true.
     \begin{itemize}
        \item [(i)] Positivity is preserved: If $f_0\in X^+$, then $f(t)\in X^+$ for all $t\in J$.
         \item [(ii)] Total mass is preserved: $m(f(t))=m(f_0)$ for all $t\in J$.
         \item [(iii)] If $X=\wloz$, then total charge is preserved:  $q(f(t))=q(f_0)$ for all $t\in J$.
        \item [(iv)] If $K(k,l)=K(-l,-k)$ for all $k,l\in\Z$, then symmetry is preserved. 
     \end{itemize}
\end{proposition}
%*****************************************************************
\begin{proof}
 (i) For $k\in\Z$, let $e_k\in X$ be given by $e_k(l)=\delta_{lk}$ (Kronecker symbol), $l\in \iZ$. Observe that $Q$ is quasi-positive, that is, for all $g\in X^+$ and $k\in \iZ$ we have that $g(k)=0$ implies $Q(g)(k)\geq0$. Hence, the claim follows from Proposition~\ref{quasipositivitypositivitypreserving} and Proposition \ref{basicpropQ} (iii). Note that $\sum_{k\in\Z}y_k e_k=y$ for all $y\in X$. 

(ii) Applying Proposition \ref{partInt} with $c(k)=1$,
we see that $\sum_{k\in\Z}Q(g)(k)=0$ for all $g\in \loz$, and thus
    \begin{align*}
        \frac{d}{dt}m(f(t))=\frac{d}{dt}\sum_{k\in\Z}f(t,k)=\sum_{k\in\Z}\partial_t f(t,k)=\sum_{k\in\Z}Q(f(t))(k)=0,\quad t\in J,
    \end{align*}
    as $f\in C^1(J;\loz)$ is a solution of \eqref{P}.
    Here we may interchange summation and differentiation by Lebesgue's theorem on differentiation of parameter integrals.

(iii) Employing Proposition \ref{partInt} with $c(k)=k$, $k\in \iZ$, we see that for every $g\in \wloz$
\[
\sum_{k\in\Z}kQ(g)(k)=\sum_{k,l\in\Z}g(k)g(l)\big( K(l,k)-K(k,l)\big)=0.
\]
Consequently,
    \begin{align*}
        \frac{d}{dt}q(f(t))=\frac{d}{dt}\sum_{k\in\Z}kf(t,k)=\sum_{k\in\Z}k\partial_t f(t,k)=\sum_{k\in\Z}kQ(f(t))(k)=0,\quad t\in J,
    \end{align*}
by the solution property of $f$, now in the space $X=\wloz$.    

(iv) Assume that $K(k,l)=K(-l,-k)$ for all $k,l\in\Z$ and let $h\in \loz$. Setting $g(k)\vcentcolon=h(-k)$, $k\in \iZ$, we claim that 
\begin{equation} \label{symmrule}
Q(h)(-k)=Q(g)(k),\quad k\in \iZ.
\end{equation}
Indeed, we have for $k\in \iZ$
\begin{align*}
        Q(h)(-k)
            &=\sum_{l\in\Z}\Big(K(l,-k-1)h(l)h(-k-1)-K(-k,l-1)h(-k)h(l-1)\\&\quad\quad-K(l,-k)h(l)h(-k)+K(-k+1,l-1)h(-k+1)h(l-1)\Big)\\
            &=\sum_{l\in\Z}\Big(K(l,-k-1)g(-l)g(k+1)-K(-k,l-1)g(k)g(-l+1)\\&\quad\quad-K(l,-k)g(-l)g(k)+K(-k+1,l-1)g(k-1)g(-l+1)\Big)\\
            &=\sum_{l\in\Z}\Big(K(k+1,-l)g(-l)g(k+1)-K(-l+1,k)g(k)g(-l+1)\\&\quad\quad-K(k,-l)g(-l)g(k)+K(-l+1,k-1)g(k-1)g(-l+1)\Big)\\
            &=Q(g)(k).
        \end{align*}
Suppose now that $f_0\in X$ is symmetric. We show that $g(t,k)\vcentcolon=f(t,-k)$, $t\in J$, $k\in \iZ$, is also a solution of \eqref{P}. Clearly,
$g(0,k)=f(0,-k)=f_0(-k)=f_0(k)$, by symmetry of $f_0$. Further,
\begin{align*}
    \partial_t g(t,k)=\partial_t f(t,-k)=Q(f(t))(-k)=Q(g(t))(k),\quad t\in J,\,k\in \iZ,
\end{align*}
where we use \eqref{symmrule}. By uniqueness, it follows that $g=f$.
Hence $f$ has the claimed symmetry property.
\end{proof}

Our next objective is to prove global existence for nonnegative initial data. For $X=\loz$, we immediately obtain this result by
combining Theorem \ref{LocalWP} and Proposition \ref{preservationprop}.
%*******************************************************
\begin{cor} \label{l1globex}
    Let $f_0\in \lozp$. Then the corresponding ($\loz$-) solution $f$ of \eqref{P} exists for all $t\ge 0$, is nonnegative, and $m(f(t))=m(f_0)$, $t\in [0,\infty)$. 
\end{cor}
%********************************************************
\begin{proof}
    Preservation of positivity and of total mass ensure that $f(t)$ does not blow up in $\loz$ in finite time. Theorem \ref{LocalWP}(ii) then implies existence for all $t\ge 0$.  
\end{proof}

Since for $X=\wloz$, the preservation properties do not provide an a priori bound in $X$, we need further consideration to get a suitable estimate for the absolute total charge of $f$.
%*************************************************************
\begin{lemma}\label{ddtabsqineq}
    Let $f_0\in\wlozp$ and $f$ be the corresponding ($\wloz$-) solution to \eqref{P} on $[0,t^+(f_0))$. Then
\begin{align} \label{l11estimate}
        \lno f(t)\rno_{1,1}\leq\lno f_0\rnooo+2C_Km(f_0)^2t,\quad t\in [0,t^+(f_0)).
    \end{align}
\end{lemma}
%*****************************************************
\begin{proof}
    By Proposition \ref{partInt} with $c(k)=|k|$ and using preservation of positivity and total mass we find that for $t\in J\vcentcolon=[0,t^+(f_0))$
    \begin{align*}
        \Big|\frac{d}{dt} &\lno f(t)\rno_{1,1}\Big|=\Big|\sum_{k\in\Z}|k|Q(f(t))(k)\Big|\\&=\Big|\sum_{k,l\in\Z}f(t,k)f(t,l)\Big(K(l,k)(|k+1|-|k|)+K(k,l)(|k-1|-|k|)\Big)\Big|\\
        &\leq \sum_{k,l\in\Z}f(t,k)f(t,l)\Big(K(l,k)\big||k+1|-|k|\big|+K(k,l)\big||k-1|-|k|\big|\Big)\\
        &\leq \sum_{k,l\in\Z}2C_K f(t,k)f(t,l)=2C_Km(f_0)^2,
    \end{align*}
    where we may differentiate under the sum since $f\in C^1(J;\wloz)$ solves \eqref{P}. With this bound on the time derivative 
    of $\lno f(t)\rno_{1,1}$, the assertion
    \eqref{l11estimate} follows directly by integration.
\end{proof}
%**********************************************************
Lemma \ref{ddtabsqineq} shows that for nonnegative initial data in $\wloz$, the associated solution $f$ cannot blow up in $\wloz$ in finite time. So, together with Theorem \ref{LocalWP}, we obtain the following main result of this section.
%*******************************************************
\begin{theorem} \label{l11globex}
    Let $f_0\in \wlozp$. Then \eqref{P} possesses a unique $\wloz$-solution $f$, which exists for all $t\ge 0$, is nonnegative, and satisfies $m(f(t))=m(f_0)$ and $q(f(t))=q(f_0)$, $t\in [0,\infty)$. 
\end{theorem}
%********************************************************

At the end of the section we look at the scaling properties of equation \eqref{ODE}.
%*************************************************************
\begin{re} \label{scaling}
Suppose that $f$ is a solution of \eqref{ODE} on the interval $[0,T]$. For $\lambda>0$, set $K_\lambda(k,l)\vcentcolon=\lambda K(k,l)$, $k,l\in \iZ$, and denote the collision operator associated with $K_\lambda$ by $Q_\lambda$. For $\gamma,\lambda>0$ define
\[
g(t,k)=\frac{1}{\lambda}f(\gamma t,k),\quad t\in \big[0,\frac{T}{\gamma}\big],\,k\in \iZ.
\]
Then it is easy to check that
\[
\partial_t g(t,k)=Q_{\gamma \lambda}(g(t,\cdot))(k),\quad t\in \big[0,\frac{T}{\gamma}\big],\,k\in \iZ.
\]
In particular, if $\gamma \lambda=1$, then $g$ is again a solution of \eqref{ODE}.  
\end{re}
%*************************************************************
\section{Lower bounds and strict positivity}
%************************************************************
The purpose of this section is to establish lower bounds for solutions, which in particular imply instantaneous strict positivity.
%**************************************************************
\begin{proposition}\label{propfdec}
    Assume {\bf ($\mathcal{B}$)}. Let $t_0\geq0$, $f_0\in \lozp$ and $f$ be the corresponding (nonnegative) solution to \eqref{P}. Then
    \begin{align*}
        f(t,k)\geq e^{-2C_Km(f_0)(t-t_0)}f(t_0,k),\quad t\ge t_0,\,k\in \Z.
    \end{align*}
\end{proposition}
\begin{proof}
    Mass conservation and assumption ($\mathcal{B}$) imply
    \begin{align*}
        \partial_t f(t,k)=Q(f(t))(k)\geq -2C_Km(f_0)f(t,k)
    \end{align*}
    for every $t\geq0$ and all $k\in\Z$, which, by Gronwall's lemma, yields the assertion.
\end{proof}
This means that if $f(t,k)>0$ for some $t\geq0$, then $f(s,k)>0$ for all $s\geq t$.
%******************************************************************
\begin{proposition}\label{propfg0}
Assume {\bf ($\mathcal{B}$)}. Let $f_0\in \lozp$ with $m(f_0)>0$ and $f$ be the corresponding solution to \eqref{P}. Then the following statements hold true.
\begin{enumerate}[(i)]
    \item\label{propfg0i} If condition {\bf ($\mathcal{P}$)} holds, then $f(t,k)>0$ for all $t>0$ and all $k\in\Z$.
    \item\label{propfg0ii} Assume in addition that $K(k,l)\geq c_K$ for all $k,l\in\Z$ with some $c_K>0$. Let $(a_n)_{n\in\N}\subset(0,\infty)$ such that $\sigma\vcentcolon=\sum_{j=1}^\infty a_j<\infty$, $\sigma_n\vcentcolon=\sum_{j=1}^n a_j$, $n\in \N$, $\sigma_0\vcentcolon=0$, $t_0>0$ and let $f_0(k)>0$ for some fixed $k\in \Z$. Then
        \begin{align}\label{fgeq1eq}
            f(t,k+l)\geq \lk c_K m(f_0)\rk^{|l|}f_0(k)e^{-2C_Km(f_0)t}\prod_{j=1}^{|l|}\lk 2C_K+\frac{\sigma}{t_0a_{j}}\rk^{-1}
        \end{align}
        for all $l\in\Z$ and all $t\geq t_0\frac{\sigma_{|l|}}{\sigma}$. 
        If in addition $(a_n)_{n\in\N}$ is monotonically nonincreasing, then in particular one has
        \begin{align}\label{fgeq2eq}
            f(t,k+l)\geq \lk c_K m(f_0)\rk^{|l|}f_0(k)e^{-2C_Km(f_0)t}\lk 2C_K+\frac{\sigma}{t_0a_{|l|}}\rk^{-|l|}
        \end{align}
        for all $l\in\Z$ and all $t\geq t_0\frac{\sigma_{|l|}}{\sigma}$.
\end{enumerate}
\end{proposition}
%***********************************************************************
\begin{proof}
We only prove \eqref{propfg0ii} as \eqref{propfg0i} can be obtained by a similar but simpler argument.
    For $l=0$ \eqref{fgeq1eq} is satisfied for all $t\geq0$ by Proposition~\ref{propfdec}. 
    Now let $l\in\N$ and assume that \eqref{fgeq1eq} holds for $l-1$ (for all $t\geq t_0\frac{\sigma_{l-1}}{\sigma}$). Then either there exists $t_1\in\lke 0,t_0\frac{\sigma_{l}}{\sigma}\rk$ such that \eqref{fgeq1eq} holds for $l$ and $t_1$, which by Proposition~\ref{propfdec} implies that \eqref{fgeq1eq} holds for $l$ and all $t\geq t_1$, or there is no such $t_1$. If the latter is the case, then
    \begin{align}\label{flles}
    f(t,k+l)< \lk c_K m(f_0)\rk^{|l|}f_0(k)e^{-2C_Km(f_0)t}\prod_{j=1}^{|l|}\lk 2C_K+\frac{\sigma}{t_0a_{j}}\rk^{-1}
        \end{align}
        for all $t\in \lke 0,t_0\frac{\sigma_{l}}{\sigma}\rk$. We will show that, in this case, \eqref{fgeq1eq} holds for $l$ and $t=t_0\frac{\sigma_{l}}{\sigma}$ and thus for all $t\geq t_0\frac{\sigma_{l}}{\sigma}$. We have
    \begin{align*}
        f(t_0\sigma_{l}&/\sigma,k+l)=\int_{t_0\sigma_{l-1}/\sigma}^{t_0\sigma_{l}/\sigma} Q(f)(\tau,k+l)\,d\tau+f(t_0\sigma_{l-1}/\sigma,k+l)\\ 
        &\geq \int_{t_0\sigma_{l-1}/\sigma}^{t_0\sigma_{l}/\sigma} \Big(c_Km(f_0)f(\tau,k+l-1)-2C_Kf(\tau,k+l)\Big)\,d\tau\\
        &\hspace{-0.9cm}\overset{\eqref{fgeq1eq}\text{ for  }l-1,\eqref{flles}}{\geq} \int_{t_0\sigma_{l-1}/\sigma}^{t_0\sigma_{l}/\sigma}\Big[ c_Km(f_0)\frac{\lk c_Km(f_0)\rk^{l-1}f_0(k)}{
        e^{2C_Km(f_0)\tau}}\prod_{j=1}^{l-1}\lk 2C_K+\frac{\sigma}{t_0a_{j}}\rk^{-1}\\
        &\quad-2C_K\lk c_Km(f_0)\rk^{l}f_0(k)e^{-2C_Km(f_0)\tau}\prod_{j=1}^{l}\lk 2C_K+\frac{\sigma}{t_0a_{j}}\rk^{-1}\Big]\,d\tau\\
        &\geq \lk c_Km(f_0)\rk^{l}f_0(k)\prod_{j=1}^{l-1}\lk 2C_K+\frac{\sigma}{t_0a_{j}}\rk^{-1}     \lk 1-\frac{2C_K}{2C_K+\frac{\sigma}{t_0a_l}}\rk\\
        &\quad \cdot\int_{t_0\sigma_{l-1}/\sigma}^{t_0\sigma_{l}/\sigma}e^{-2C_Km(f_0)\tau}\,d\tau\\
        &\geq \lk c_Km(f_0)\rk^{l}f_0(k)\prod_{j=1}^{l-1}\lk 2C_K+\frac{\sigma}{t_0a_{j}}\rk^{-1}     \lk 1-\frac{2C_K}{2C_K+\frac{\sigma}{t_0a_l}}\rk\\
        &\quad  \cdot t_0\frac{a_l}{\sigma}e^{-2C_Km(f_0)t_0\sigma_{l}/\sigma}\\
        &=\lk c_Km(f_0)\rk^{l}f_0(k)\prod_{j=1}^{l}\lk 2C_K+\frac{\sigma}{t_0a_{j}}\rk^{-1}e^{-2C_Km(f_0)t_0\sigma_{l}/\sigma},
    \end{align*}
    which is what we wanted. The case $l<0$ is proven analogously.
    
    In order to derive \eqref{fgeq2eq} it suffices to show that 
    \begin{align*}
        \lk 2C_K+\frac{s}{t_0a_{j}}\rk^{-1}\geq \lk 2C_K+\frac{s}{t_0a_{|l|}}\rk^{-1}
    \end{align*}
    for all $j\in\{1,\dots,|l|\}$. But this is a direct consequence of the monotonicity of $(a_n)_{n\in\N}$.
\end{proof}

\begin{re}\label{reequipos}
Assuming that {\bf ($\mathcal{P}$)} holds, an important consequence of Proposition~\ref{propfg0} \eqref{propfg0i} is that any equilibrium
$f_*\in\lozp$ of \eqref{P} with positive mass satisfies $f_*(k)>0$ for all $k\in\Z$.
\end{re}
%*******************************************************************
%*************************************************************
\section{Approximation by truncation}
%**************************************************************
In this section, we show that, on any compact time interval $[0,T]$, every nonnegative $\wloz$-solution of \eqref{P} can be approximated arbitrarily well by solutions to related finite-dimensional problems.
This property is useful for several technical reasons, in particular, it will be employed in Section \ref{EntropySec} to establish the monotonicity of relative entropy. 

The basic idea of the approximation is to consider truncated versions of \eqref{P}, a technique already used in the EDG case \cite{EDG}. In contrast to \cite[Section 2.1]{EDG}, our convergence result, Theorem \ref{truncapproxc10T} below, yields convergence of the entire sequence,
without the need to extract a subsequence. This is possible in our setting because we work with bounded kernels.
%********************************************************************

For $N\in\N$, we set $S_N\vcentcolon=\{k\in\Z:|k|\leq N\}$. Given a function $g:\,\Z\to \R$, we define its truncation $g^N:\,\Z\to \R$ by $g^N(k)=g(k)$ for all $k\in S_N$ and $g^N(k)=0$ for all $k\in \Z\setminus S_N$.
Moreover, for $N\in\N$ and a kernel $K:\,\Z^2\rightarrow[0,\infty)$, we define its truncation $K^N:\,\Z^2\rightarrow[0,\infty)$ by
\begin{align*}
            K^N(k,l)\vcentcolon=\begin{cases}
            K(k,l), &\text{ if }k\in\{-N+1,\dots,N\}, l\in\{-N,\dots,N-1\}\\
            0, &\text{ otherwise.}
        \end{cases}
        \end{align*}
We denote by $Q^N$ the collision operator associated with $K^N$ and also set $Q^N(f)=Q^N(f,f)$.         
%*****************************************************************

Given $f_0\in \wlozp$ we consider the truncated systems  
\begin{equation}\label{IVPN}
    \begin{cases}
        \dot{f}(t)&=\,Q^N(f(t)),\quad t>0,\\
        f(0)&=\,f_0^N,
    \end{cases}\tag{PN}
\end{equation}
for $N\in \N$. We can apply Theorem \ref{l11globex} to the (bounded) kernel $K^N$, which shows that \eqref{IVPN} has a unique $\wloz$-solution $f_N$ on $[0,\infty)$, which is nonnegative and satisfies $m(f_N(t))=m(f_0^N)$ and $q(f_N(t))=q(f_0^N)$ for all $t\in [0,\infty)$. Observe that, by definition of $Q^N$,  $f_N(t,k)=0$ for all
$t\ge 0$ and $k\in \Z\setminus S_N$. Hence \eqref{IVPN} may also be regarded as a system of $2N+1$ ODEs.
%****************************************************************
\begin{theorem}\label{truncapproxc10T}
    Assume {\bf ($\mathcal{B}$)}. Let $f_0\in\wlozp$ and $f$ be the corresponding solution to \eqref{P}. For $N\in \N$ let $f_N$ be the solution to \eqref{IVPN} with initial value $f_0^N$. Then for every $T>0$, 
    \begin{align*}
        \lim_{N\rightarrow\infty}\lno f_N-f\rno_{C^1([0,T],\wloz)}=0.
    \end{align*}
\end{theorem}
%*************************************************************
\begin{proof}
The assertion is evident for $f_0=0$ since then $f(t)=f_N(t)$ for all $t\ge 0$ and $N\in \N$. For $f_0\in\wlozp\setminus\{0\}$, by scaling, we may assume that $m(f_0)=1$. This implies $m(f_N(t))=m(f_0^N)\le 1$  for all $t\ge 0$ and $N\in \N$.

Let $N\in\N$ and $g:\,\Z\to [0,\infty)$ with $g(k)=0$ for all $k\in \Z\setminus S_N$. We claim that
\begin{align} \label{QdiffEst}
        \lno Q^N(g)-Q(g)\rnooo\leq (6+4N)C_K m(g)\big( g(-N)+g(N)\big).
    \end{align}
Indeed, we have
    \begin{align*}
        \lv K^N(i,j)-K(i,j)\rv g(i)g(j)=\begin{cases}
            K(i,j)g(i)g(j), &\text{ if }(i=-N \text{ and }j\in S_N)\\&\quad\text{ or }(j=N \text{ and } i\in S_N)\\
            0, &\text{ otherwise},
        \end{cases}
    \end{align*}
    and thus
    \begin{align*}
        &\lno Q^N(g)-Q(g)\rnooo\\&\quad\leq \sum_{k,l\in\Z}(1+|k|)\Big[ \lv K^N(l,k-1)-K(l,k-1) \rv g(l)g(k-1) \\&\quad\quad\quad\quad\quad+\lv K^N(k,l-1)-K(k,l-1) \rv g(k)g(l-1)\\&\quad\quad\quad\quad\quad+\lv K^N(l,k)-K(l,k) \rv g(l)g(k)\\ &\quad\quad\quad\quad\quad+\lv K^N(k+1,l-1)-K(k+1,l-1) \rv g(k+1)g(l-1) \Big]\\
        &\quad\leq \sum_{k,l\in\Z}(3+2|k|)\Big[\lv K^N(l,k)-K(l,k) \rv g(l)g(k)\\
        &\quad\quad\quad\quad\quad+\lv K^N(k,l)-K(k,l) \rv g(k)g(l)\Big]\\
        &\quad\leq 2(3+2N)\sum_{k,l\in S_N}\lv K^N(k,l)-K(k,l) \rv g(k)g(l)\\
        &\quad\leq (6+4N)\Big[ \sum_{l\in S_N}K(-N,l)g(-N)g(l)+\sum_{k\in S_N} K(k,N)g(k)g(N)\Big]\\
        &\quad\leq (6+4N)C_K m(g)\big( g(-N)+g(N)\big).
    \end{align*}
    
Let now $t\ge 0$. Applying estimate \eqref{QdiffEst} to $g=f_N(t)$ and using Proposition \ref{basicpropQ}, we obtain
    \begin{align}
        &\lno Q^N(f_N(t))-Q(f(t))\rnooo\nonumber\\&\quad\leq \lno Q^N(f_N(t))-Q(f_N(t))\rnooo+\lno Q(f_N(t))-Q(f(t))\rnooo\nonumber\\
        &\quad\leq (6+4N)C_K \lke f_N(t,-N)+f_N(t,N)\rke\nonumber\\&\quad\quad+6C_K\lk \lno f_N(t)\rnooo+\lno f(t)\rnooo\rk\lno f_N(t)-f(t)\rnooo\nonumber\\
        &\quad\leq (6+4N)C_K\Big[ f(t,-N)+f(t,N)+\lv f_N(t,-N)-f(t,-N)\rv\nonumber\\
        &\quad\quad\quad\quad\quad\quad\quad\quad+\lv f_N(t,N)-f(t,N)\rv\Big]\nonumber\\&\quad\quad+6C_K\lk \lno f_N(t)\rnooo+\lno f(t)\rnooo\rk\lno f_N(t)-f(t)\rnooo\nonumber\\
        &\quad\leq (6+4N)C_K\big(f(t,-N)+f(t,N)\big)\nonumber\\&\quad\quad+6C_K\lk \lno f_N(t)\rnooo+\lno f(t)\rnooo +1\rk\lno f_N(t)-f(t)\rnooo. \label{Qdiffest2}
    \end{align}

    Now, let $T>0$ be fixed. Applying Lemma~\ref{ddtabsqineq} to \eqref{P} and \eqref{IVPN} and recalling that the total mass of $f(t)$ and $f_N(t)$ is bounded by one for all $t\ge 0$, we obtain that
    \begin{align}
        \lno f_N(t)\rnooo&\leq \lno f_0\rnooo+2C_Kt,\label{Abound1}\\
        \lno f(t)\rnooo&\leq \lno f_0\rnooo+2C_Kt, \label{Abound2}
    \end{align}
    for all $t\geq0$.
    Moreover, since $f$ is continuous, $f([0,T])$ is compact in $\wloz$. Thus, by Proposition~\ref{valleepoussin}, there is a superlinear sequence $x:\,\N_0\to (0,\infty)$ such that $$C_x\vcentcolon=\sup_{t\in[0,T]}\sum_{k\in\Z} x(|k|)f(t,k)<\infty.$$ This implies 
    $$f(t,N)+f(t,-N)\leq \frac{C_x}{x(N)},\quad t\in [0,T],\,N\in \N.$$ Plugging \eqref{Abound1}, \eqref{Abound2} and the last estimate into \eqref{Qdiffest2} we obtain that, for $t\in [0,T]$, 
    \begin{align}
        &\lno Q^N(f_N(t))-Q(f(t))\rnooo\nonumber\\
        &\quad\leq 6C_K\lk(1+N)\frac{C_x}{x(N)}+\lk 2\lno f_0\rnooo+4C_KT+1\rk\lno f_N(t)-f(t)\rnooo\rk.\label{QdiffEst3}
    \end{align}
    Using \eqref{QdiffEst3} and the integral formulations of \eqref{P} and \eqref{IVPN}, we estimate as follows for $t\in [0,T]$.
    \begin{align*}
        &\lno f_N(t)-f(t)\rnooo\leq \sum_{k\in\Z}(1+|k|)|f_N(t,k)-f(t,k)|\\&\quad=\sum_{k\in\Z}(1+|k|)\lv f_0^N(k)-f_0(k)+\int_0^t \big[Q^N(f_N(s))(k)-Q(f(s))(k)\big]\,ds\rv\\&\quad\leq \lno f_0^N-f_0\rnooo+\int_0^t\lno Q^N(f_N(s))-Q(f(s))\rnooo\,ds\\
        &\quad\leq \lno f_0^N-f_0\rnooo+6C_K(1+N)\frac{C_x}{x(N)}T\\
        &\quad\quad+6C_K\lk 2\lno f_0\rnooo+4C_KT+1\rk\int_0^t\lno f_N(s)-f(s)\rnooo\,ds.
    \end{align*}
    Setting $\eta=6C_K\lk 2\lno f_0\rnooo+4C_KT+1\rk$, it follows by the Gronwall lemma that
    \begin{align*}
        \lno f_N(t)-f(t)\rnooo \leq \lk \lno f_0^N-f_0\rnooo+6C_K(1+N)\frac{C_x}{x(N)}T\rk e^{\eta T}.
    \end{align*}
    Since $\lno f_0^N-f_0\rnooo\to 0$ and 
    $(1+N)/x(N)\rightarrow0$ as $N\rightarrow\infty$, we infer that 
    \begin{align}
        \sup_{t\in[0,T]}\lno f_N(t)-f(t)\rnooo\rightarrow0\quad \text{as}\;\, N\rightarrow\infty.\label{UnifLim}
    \end{align}

Finally, using the evolution equations for $f_N$ and $f$ and applying the estimates \eqref{QdiffEst3}, \eqref{UnifLim} we find that
\begin{align*}
\sup_{t\in[0,T]}\lno \dot{f}_N(t)-\dot{f}(t)\rnooo=\sup_{t\in[0,T]}\lno Q^N(f_N(t))-Q(f(t))\rnooo \to 0\quad\text{as}\;\,N\to \infty,  
\end{align*}
which completes the proof.
\end{proof}
%***********************************************
\begin{cor}\label{truncapproxmass1}
    Assume {\bf ($\mathcal{B}$)}. Let $f_0\in\wlozp$ with $m(f_0)=1$ and $f$ be the corresponding solution to \eqref{P}. Fix $N_0\in \N$ such that $m(f_0^{N_0})>0$. For $N\in \N$ with $N\ge N_0$, let $\tilde{f}_N$ be the solution to \eqref{IVPN} with initial value $\tilde{f}_0^N\vcentcolon=f_0^N/m(f_0^N)$. Then for every $T>0$, 
    \begin{align*}
        \lim_{N\rightarrow\infty}\lno \tilde{f}_N-f\rno_{C^1([0,T],\wloz)}=0.
    \end{align*}
\end{cor}
%*************************************************
\begin{proof}
    For $N\ge N_0$, set $\gamma_N=m(f_0^{N})^{-1}$ and let $f_N$ be the solution of \eqref{IVPN} with initial value $f_0^N$. By scaling (see Remark \ref{scaling}) and uniqueness,
    \[
    \tilde{f}_N(t,\cdot)=\gamma_N f_N(\gamma_N t,\cdot),\quad t\ge 0.
    \]
    Since $\gamma_N\to 1$, it follows that 
    $\lim_{N\rightarrow\infty}\lno \tilde{f}_N-f_N\rno_{C^1([0,T],\wloz)}=0$ for all $T>0$. Combined with Theorem \ref{truncapproxc10T}, this yields the claim.  
\end{proof}

%*************************************************************
\section{Detailed balance and equilibria}
%**********************************************************
Throughout this section we assume ($\mathcal{P}$). However,
condition ($\mathcal{B}$) is not required.

As already mentioned in the introduction, detailed balance plays an important role in kinetic equations and chemical reaction systems. We recall that condition
($\mathcal{DB}$) states that the kernel $K:\Z^2\rightarrow(0,\infty)$ admits detailed balancing, which means that there exists a function $g:\,\iZ\to (0,\infty)$ such that
    \begin{equation}\label{db}
        K(k,l-1)g(k)g(l-1)=K(l,k-1)g(l)g(k-1)\quad \mbox{for all}\;k,l\in \iZ.
    \end{equation}
In this case, $g$ is said to satisfy the {\em detailed balance condition}.

%*************************************************************
\begin{re} \label{DBRemark}
(i) If $g\in \loz$ satisfies the detailed balance condition and all series in the definition of $Q(g)$ converge, then
$Q(g)=0$, that is, $g$ is an equilibrium of \eqref{ODE}. Indeed, for every $k\in \iZ$, we have
\begin{align*}
    Q(g)(k)&=\sum_{l\in\Z}\underbrace{\big(K(l,k-1)g(l)g(k-1)-K(k,l-1)g(k)g(l-1)\big)}_{=0}\\&\quad\quad\quad+\sum_{l\in\Z}\underbrace{\big(-K(l,k)g(l)g(k)+K(k+1,l-1)g(k+1)g(l-1)\big)}_{=0}=0.
\end{align*}
(ii) Note that if $g:\,\iZ\to \iR$ satisfies \eqref{db} and 
there is $i\in\iZ$ such that $g(i)=0$, then $g=0$ on $\iZ$. This can be seen as follows. First, by \eqref{db}, we have
\[
K(i+1,i+1)g(i+1)^2=K(i+2,i)g(i+2)g(i)=0.
\]
Since $K>0$, it follows that $g(i+1)=0$. Analogously, one obtains
that $g(i-1)=0$. The claim then follows by induction. This is also consistent with Remark~\ref{reequipos}.
\end{re}
%**************************************************************

We will now examine the structure of the functions satisfying the detailed balance condition in more detail.

Suppose $g:\,\iZ\to (0,\infty)$ satisfies \eqref{db}. 
Applying \eqref{db} with $k=0,l=1$ gives
\begin{align}\label{g0}
    g(0)=\kappa \sqrt{g(1)g(-1)}\;\quad\; \mbox{with}\;\;
    \kappa\vcentcolon=\sqrt{\frac{K(1,-1)}{K(0,0)}}.
\end{align}
Next, solving \eqref{db} for $g(k)$ yields
\begin{align}\label{dbceqsearch1}
    g(k)=\frac{K(l,k-1)}{K(k,l-1)}\frac{g(l)g(k-1)}{g(l-1)}.
\end{align}
In particular, for $l=1$ we obtain
\begin{align} \label{posrec}
    g(k)=\frac{K(1,k-1)}{K(k,0)}\frac{g(1)g(k-1)}{g(0)}.
\end{align}
Analogously, solving \eqref{db} for $g(k-1)$ and replacing $k-1$ by $k$, the case $l=0$ yields
\begin{align} \label{negrec}
    g(k)=\frac{K(k+1,-1)}{K(0,k)}\frac{g(k+1)g(-1)}{g(0)}.
\end{align}
Setting
\begin{align}
    \Tilde{\psi}(k)\vcentcolon=\begin{cases}
        \prod_{j=1}^k\frac{K(1,j-1)}{K(j,0)},\, &k>0\\
        1, &k=0\\
        \prod_{j=k}^{-1}\frac{K(j+1,-1)}{K(0,j)},\, &k<0,
    \end{cases}
\end{align}
and iteratively applying \eqref{posrec} and \eqref{negrec} we obtain
\begin{align} \label{gk1}
    g(k)&=\Tilde{\psi}(k)\lk\frac{g(1)}{g(0)}\rk^k g(0), \quad k\geq0,\\
    g(k)&=\Tilde{\psi}(k)\lk\frac{g(-1)}{g(0)}\rk^{|k|}g(0),\quad k\leq0.
\end{align}
Using \eqref{g0} these relations are equivalent to
\begin{align}\label{gk2}
    g(k)&=\Tilde{\psi}(k)\kappa^{-|k|}\Big(\frac{g(1)}{g(-1)}\Big)^\frac{k}{2} g(0), \quad k\in \iZ,
\end{align}
where we can also express $g(0)$ in terms of $g(1)$ and $g(-1)$ employing \eqref{g0}.

In order to better understand when a positive function $g$ subject to
\eqref{gk2} belongs to $\loz$ and $\wloz$, respectively, and to ensure the existence of detailed balanced equilibria in these spaces we impose the assumption {\bf ($\mathcal{E}$)} (cf.\ the introduction), which states that
the limits    
    \begin{align*}
        \Lambda_+\vcentcolon=\lim_{j\rightarrow\infty}\frac{K(j,0)}{K(1,j-1)}\in (0,\infty],\quad 
        \Lambda_-\vcentcolon=\lim_{j\rightarrow-\infty}\frac{K(0,j)}{K(j+1,-1)}\in (0,\infty]
    \end{align*}
    exist and $ \mathbf{I}_K\vcentcolon=\lk\lk \Lambda_-\kappa\rk^{-1},\Lambda_+\kappa\rk\neq\emptyset$.

Observe that 
\begin{align} \label{psitildeinfty}
    \lim_{k\rightarrow\infty} \Tilde{\psi}(k)^{1/k}=\Lambda_+^{-1} \;\text{ and }\;
    \lim_{k\rightarrow-\infty} \Tilde{\psi}(k)^{1/|k|}=\Lambda_-^{-1},
\end{align}
where $\frac{1}{\infty}\vcentcolon=0.$

Thus, for $g$ as above, we obtain that if $\sqrt{\frac{g(1)}{g(-1)}}\in  \mathbf{I}_K$, then $g\in\wloz$, and if $\sqrt{\frac{g(1)}{g(-1)}}\not\in \overline{ \mathbf{I}_K}$, then $g\not\in\loz$. 

\begin{definition}\label{deffphi}
    Set $\psi(k)\vcentcolon=\Tilde{\psi}(k)\kappa^{-|k|}$, $k\in \iZ$. Define $Z:\overline{\mathbf{I}_K} \cap(0,\infty)\rightarrow[0,\infty]$ by
    \[
    Z(\phi)\vcentcolon=\sum_{l\in\Z}\psi(l)\phi^l
    \]
    and the map $\Phi: \mathbf{I}_K\rightarrow\wloz,\phi\mapsto f_\phi$, where $f_\phi$ is given by
    \begin{align*}
        f_\phi(k)\vcentcolon=\frac{\psi(k)\phi^k}{Z(\phi)},\quad k\in \iZ.
    \end{align*}
\end{definition}
%***********************************************************
\begin{re}\label{reequiwell}
    (i) By the preceding considerations, $f_\phi$ is indeed well-defined and and an element of $\wloz$ with unit mass for all $\phi\in  \mathbf{I}_K$.

    (ii) Note that if $\Lambda_+,\Lambda_-\in (0,\infty)$, then $f_\phi$ is also well-defined and an element of $\loz$ with unit mass for $\phi\in \{\lk \Lambda_-\kappa\rk^{-1},\Lambda_+\kappa\}$ provided that $Z(\phi)<\infty$.

     (iii) Let $\Lambda_+\in(0,\infty)$ and $\phi\in  \mathbf{I}_K$. Then for every $\varepsilon>0$ with $(1+\varepsilon)(\kappa\Lambda_+)^{-1}\phi<1$ (which is always possible, since $\phi<\kappa\Lambda_+$) there exists $C>0$ such that 
     \begin{align*}
         &f_\phi(k)\leq C\lk(1+\varepsilon)(\kappa\Lambda_+)^{-1}\phi\rk^k \text{ for all } k\in\N_0 \text{ and }\\
         &\sum_{k\in\N_0} C\lk(1+\varepsilon)(\kappa\Lambda_+)^{-1}\phi\rk^k<\infty
     \end{align*} This is a direct consequence of \eqref{psitildeinfty}. Moreover, for every $\varepsilon\in(0,1)$ there exists $c>0$ such that $$f_\phi(k)\geq c\lk(1-\varepsilon)(\kappa\Lambda_+)^{-1}\phi\rk^k \text{ for all }k\in\N_0.$$
     Analogue statements hold for $\Lambda_-$.
\end{re}
%************************************************************
\begin{proposition}
Assume {\bf($\mathcal{P}$)}. The following statements are equivalent.
    \begin{itemize}
        \item [(i)] $K$ admits detailed balancing.
        \item [(ii)] $\psi$ satisfies \eqref{db}.
        \end{itemize}
If in addition {\bf($\mathcal{E}$)} is satisfied, then these statements are further equivalent to
\begin{itemize}
        \item [(iii)] For all $\phi\in  \mathbf{I}_K$ the function $f_\phi(k)$ satisfies \eqref{db}.
    \end{itemize}
\end{proposition}
%*************************************************************
\begin{proof} Note first that $\psi$ and $f_\phi$ ($\phi\in  \mathbf{I}_K$) take only values in $(0,\infty)$. Thus each of two conditions (ii) and (iii) (under ($\mathcal{E}$)) implies (i). 

Now assume that (i) holds. Then from our considerations above we know that there exists $g:\,\iZ\to (0,\infty)$ that satisfies \eqref{db} and is of the form
$g(k)=\psi(k)\eta^k g(0)$, $k\in \iZ$, with some $\eta>0$; see \eqref{gk2}. By normalization, we can assume that $g(0)=1$.
Since $g$ satisfies \eqref{db}, we have
\begin{align} \label{chain}
    \frac{K(k,l-1)}{K(l,k-1)}=\frac{g(l)g(k-1)}{g(k)g(l-1)}=
    \frac{\psi(l)\eta^l \psi(k-1)\eta^{k-1}}{\psi(k)\eta^k \psi(l-1)\eta^{l-1}}=\frac{\psi(l) \psi(k-1)}{\psi(k) \psi(l-1)}
\end{align}
for all $k,l\in \iZ$, which shows (ii). If ($\mathcal{E}$) holds, then statement (iii) follows directly from (ii) and the structure of $f_\phi$ by extending the last fraction in \eqref{chain} with $\phi^{l+k-1}Z(\phi)^{-2}$. 
\end{proof}
%**************************************************************
 \begin{definition}
        The kernel $K$ is said to satisfy the \emph{extended Becker-Döring assumption} if
        \begin{gather*}
            \begin{aligned}
            \frac{K(k,l-1)}{K(l,k-1)}&=\frac{K(k,0)K(1,l-1)}{K(l,0)K(1,k-1)} &&\text{ for all }k,l\geq1,\\
            \frac{K(k,l-1)}{K(l,k-1)}&=\frac{K(0,l-1)}{K(0,k-1)}\frac{K(k,-1)}{K(l,-1)}&&\text{ for all }k,l\leq0,\\
            \frac{K(k,l-1)}{K(l,k-1)}&=\frac{K(k,0)}{K(1,k-1)}\frac{K(0,l-1)}{K(l,-1)}\frac{K(1,-1)}{K(0,0)}&&\text{ for all }k\geq1,l\leq0\text{ and }\\
            \frac{K(k,l-1)}{K(l,k-1)}&=\frac{K(1,l-1)}{K(l,0)}\frac{K(k,-1)}{K(0,k-1)}\frac{K(0,0)}{K(1,-1)}&&\text{ for all }k\leq0,l\geq1.
        \end{aligned}
        \end{gather*}
    \end{definition}
%**************************************************************
\begin{theorem}
    Assume {\bf ($\mathcal{P}$)}. $K$ admits detailed balancing if and only if $K$ satisfies the extended Becker-Döring assumption.
\end{theorem}
%*******************************************************************************
\begin{proof}
    First, assume $K$ admits detailed balancing. Then there exists $g:\Z\rightarrow(0,\infty)$ that satisfies \eqref{db}. Hence, by the considerations at the beginning of this section, it follows that $g$ is given by \eqref{gk2}. We only show the assertion for the case $k\geq1$ and $l\leq0$, as the other cases can be treated analogously. We have
    \begin{align*}
        &K(k,l-1)g(k)g(l-1)=K(l,k-1)g(l)g(k-1)\\
        &\iff \frac{K(k,l-1)}{K(l,k-1)}=\frac{\Tilde{\psi}(k-1)\kappa^{-|k-1|}\Big(\frac{g(1)}{g(-1)}\Big)^\frac{k-1}{2}}{\Tilde{\psi}(k)\kappa^{-|k|}\Big(\frac{g(1)}{g(-1)}\Big)^\frac{k}{2}}\frac{\Tilde{\psi}(l)\kappa^{-|l|}\Big(\frac{g(1)}{g(-1)}\Big)^\frac{l}{2}}{\Tilde{\psi}(l-1)\kappa^{-|l-1|}\Big(\frac{g(1)}{g(-1)}\Big)^\frac{l-1}{2}}\\
        &\iff \frac{K(k,l-1)}{K(l,k-1)}=\frac{\Tilde{\psi}(k-1)\kappa}{\Tilde{\psi}(k)\Big(\frac{g(1)}{g(-1)}\Big)^\frac{1}{2}}\frac{\Tilde{\psi}(l)\kappa}{\Tilde{\psi}(l-1)\Big(\frac{g(1)}{g(-1)}\Big)^{-\frac{1}{2}}}\\
        &\iff\frac{K(k,l-1)}{K(l,k-1)}=\frac{\Tilde{\psi}(k-1)}{\Tilde{\psi}(k)}\frac{\Tilde{\psi}(l)}{\Tilde{\psi}(l-1)}\kappa^2\\
        &\iff \frac{K(k,l-1)}{K(l,k-1)}=\frac{K(0,l-1)}{K(l,-1)}\frac{K(k,0)}{K(1,k-1)}
        \frac{K(1,-1)}{K(0,0)},
    \end{align*}
    which is what we wanted.

    Now, assume $K$ satisfies the extended Becker-Döring assumption. Then $\psi$ satisfies \eqref{db} and thus $K$ admits detailed balancing. This can be seen by similar calculations as in the first part of this proof by substituting $g$ with $\psi$.
\end{proof}
%********************************************************************

\begin{re}
 We note that the equation for $k,\,l\geq 1$ in the extended Becker-Döring assumption coincides with the Becker-Döring assumption for the EDG model in \cite{EDG}. 
 \end{re}
 %******************************************************************
 \begin{re}
 Let $K_1,K_2:\Z^2\rightarrow(0,\infty)$. If $K_1$ and $K_2$ admit detailed balancing, then $K_1\cdot K_2$ also admits detailed balancing. This can be seen by the extended Becker-Döring assumption or by definition: if $g_1$ satisfies the detailed balance condition with respect to $K_1$ and $g_2$ satisfies the detailed balance condition with respect to $K_2$, then $g_1\cdot g_2$ satisfies the detailed balance condition with respect to $K_1\cdot K_2$. Moreover, if $K$ enjoys the detailed balance property (say with function $g$), then also $K^{-1}$ does so (consider $g^{-1}$). In particular, this means that the set of all detailed balance kernels $K:\Z^2\rightarrow(0,\infty)$ forms an abelian group, with neutral element $K_e\equiv1$. It is straightforward to verify that $\kappa_{K_1\cdot K_2}=\kappa_{K_1}\cdot\kappa_{K_2}$. The same holds for $\Lambda_\pm,\Tilde{\psi}$ and $\psi$. Hence the interval boundaries of $\mathbf{I}_{K_1\cdot K_2}$ are the products of the respective interval boundaries of $\mathbf{I}_{K_1}$ and $\mathbf{I}_{K_2}$.
\end{re}

The next lemma states that total charge is strictly increasing along the family $(f_\phi)_{\phi\in \mathbf{I}_K}$. Its proof is analogous to the argument in \cite[(1.14)]{EDG}.
%********************************************************************
\begin{lemma}\label{ddtqequi}
     Assume ($\mathcal{P}$) and {($\mathcal{E}$)}. Then the function $q(\Phi(\,\cdot\,)): \mathbf{I}_K\rightarrow\R$ is continuously differentiable and $\frac{d}{d\phi}q\lk \Phi(\phi)\rk>0$ for all $\phi\in  \mathbf{I}_K$.
\end{lemma}
\begin{proof}
We have
    \begin{align*}
        \phi\,\frac{d}{d\phi}\,q(\Phi(\phi))&=\phi\,\frac{d}{d\phi}\,\frac{\sum_{k\in\Z} k\psi(k)\phi^k}{\sum_{k\in\Z}\psi(k)\phi^k}\\
        &=\frac{\sum_{k\in\Z} k^2\psi(k)\phi^k}{\sum_{k\in\Z}\psi(k)\phi^k}-\lk\frac{\sum_{k\in\Z} k\psi(k)\phi^k}{\sum_{k\in\Z}\psi(k)\phi^k}\rk^2>0,
    \end{align*}
    where the last step is a consequence of Jensen's inequality. 
\end{proof}
%**********************************************************************************
\begin{exa}\label{Kexag}
    Let $(x_k)_{k\in\Z},(z_k)_{k\in\Z}\subset (0,\infty)^{\Z}$, $c>0$ and let $S:\Z^2\rightarrow(0,\infty)$ satisfy $S(k,l-1)=S(l,k-1)$ for all $k,l\in\Z$. Then
    \begin{align}
        K(k,l)\vcentcolon=\begin{cases}
            c S(k,l)x_kz_l, &\text{ if }l<0<k\\
            S(k,l)x_kz_l, &\text{ otherwise},
        \end{cases}
    \end{align}
    satisfies the extended Becker-Döring assumption. See also \cite[Example 1.6]{EDG}. 
\end{exa}
%**************************************************************************************
We will now study the following special case in more detail.
%****************************************************************************************
\begin{exa}    
    Let $a,b>0$ and
    \begin{align}
        K_{a,b}(k,l)\vcentcolon=\begin{cases}
            a+b, &\text{ if }l<0<k\\
            a, &\text{ otherwise}.
        \end{cases}
    \end{align}
First, observe that when $x_k=z_k=\sqrt{a}$ for all $k\in\Z$ and $c=(a+b)/a$, and $S\equiv 1$, the kernel takes the form given in Example~\ref{Kexag}. Moreover, we have $\Tilde{\psi}(l)=1$, $l\in\Z$, $\kappa=\sqrt{\frac{a+b}{a}}$, and thus $\psi(l)=\sqrt{\frac{a}{a+b}}^{|l|}$, $l\in\Z$, and $ \mathbf{I}_K=\lk\sqrt{\frac{a}{a+b}},\sqrt{\frac{a+b}{a}}\rk=(\kappa^{-1},\kappa)$. By Lemma~\ref{ddtqequi}, the function 
$q(\Phi(\,\cdot\,)):\, \mathbf{I}_K\rightarrow  \mathbf{J}_K$
    is bijective.
    For $\phi\in \mathbf{I}_K$ we have $f_\phi(k)={\sqrt{\frac{a}{a+b}}^{|k|}\phi^k}{Z(\phi)^{-1}}$, $k\in \Z$, with
   $$Z(\phi)=\frac{1}{1-\kappa^{-1}\phi^{-1}}+\frac{1}{1-\kappa^{-1}\phi}-1,$$
    and thus
    \begin{align*}
        q(f_\phi)=\lk\frac{\kappa^{-1}\phi}{\lk1-\kappa^{-1}\phi\rk^2}-\frac{\kappa^{-1}\phi^{-1}}{\lk1-\kappa^{-1}\phi^{-1}\rk^2}\rk\lk\frac{1}{1-\kappa^{-1}\phi}+\frac{1}{1-\kappa^{-1}\phi^{-1}}-1\rk^{-1}.
    \end{align*}
    Observe that $\lim_{\phi\rightarrow \kappa-}q(f_\phi)=\infty$ and
    $\lim_{\phi\rightarrow \frac{1}{\kappa}+}q(f_\phi)=-\infty$. Hence, $ \mathbf{J}_K=\R$.
\end{exa}

%*************************************************************************************
\begin{exa}\label{ExaFinMaxCharge}
    The following example is inspired by Example~1.6 in \cite{EDG} (see also Section 1.1 in \cite{Nh02}). Let $\gamma\in (0,1)$, $c>1$ and 
    \begin{align*}
        y_k\vcentcolon=\begin{cases}
            1+1/k^\gamma,&\text{ if }k\in\N\\
            1, &\text{ otherwise}.
        \end{cases}
    \end{align*}
    Consider the kernel given by
    \begin{align}
        K(k,l)\vcentcolon=\begin{cases}
            cy_ky_{-l},&\text{ if }l<0<k\\
            y_ky_{-l}, &\text{ otherwise}.
            \end{cases}
    \end{align}
    We have $\kappa=2\sqrt{c}$, $\Lambda_\pm=\frac{1}{2}$, $ \mathbf{I}_K=(\sqrt{c}^{-1},\sqrt{c})$ and
    \begin{align*}
        \psi(k)=\begin{cases}
            \frac{1}{2\sqrt{c}^k}\prod_{j=2}^k \lk 1+\frac{1}{j^\gamma}\rk^{-1}, & k>0\\
            1, &k=0\\
            \frac{\sqrt{c}^k}{2}\prod_{j=2}^{-k}\lk1+\frac{1}{j^\gamma}\rk^{-1},& k<0,
        \end{cases}
    \end{align*}
with the convention that $\prod_{j=2}^1 a_j\vcentcolon=1$. Note that $\mathbf{J}_K$ is bounded. Consequently, for sufficiently large positive or negative total charges, there is no detailed balance equilibrium of unit mass with such a total charge.
We give an argument that $\mathbf{J}_K$ is bounded from above, the argument for the lower bound is similar. In order to prove this boundedness, we need to show that
\begin{align*}
    \lim_{\phi\rightarrow\phi_+}q(f_\phi)=\lim_{\phi\rightarrow\sqrt{c}}\frac{\skz k\psi(k)\phi^k}{\skz \psi(k)\phi^k}<\infty.
\end{align*}
Let $\Tilde{\phi}\in \mathbf{I}_K$ be fixed. Then $\skz \psi(k)\phi^k\geq\sum_{k\in\N}\psi(k)\Tilde{\phi}^k>0$ for all $\phi\in [\Tilde{\phi},\sqrt{c})$. Thus it suffices to prove that 
\begin{align*}
    \lim_{\phi\rightarrow\sqrt{c}}\skz k\psi(k)\phi^k<\infty.
\end{align*}
Let $\phi\in (1,\sqrt{c})$. Using that $\log x\ge \frac{1}{2}(x-1)$ for all $x\in [1,2]$, we have
\begin{align*}
 &\skz k\psi(k)\phi^k\leq\sum_{k\in\N}k\psi(k)\phi^k=
 \sum_{k\in\N} \frac{k}{2\sqrt{c}^k}\prod_{j=2}^k \lk 1+\frac{1}{j^\gamma}\rk^{-1} \phi^k \\
 & \le \frac{1}{2}+\frac{1}{2}\sum_{k=2}^\infty k\exp\Big(-\sum_{j=2}^k \log\Big( 1+\frac{1}{j^\gamma}\Big)
 \Big)\le \frac{1}{2}+\frac{1}{2}\sum_{k=2}^\infty k\exp\Big(-\frac{1}{2}\sum_{j=2}^k 
 \frac{1}{j^\gamma}\Big)<\infty,
\end{align*}
since $\gamma\in (0,1)$ implies that $\sum_{j=2}^k 
 \frac{1}{j^\gamma}\sim c(\gamma)k^{1-\gamma}$ for $k\to \infty$.
\end{exa}

\begin{exa}
    We revisit the case $K\equiv1$. As already mentioned in the introduction, for every initial value $f_0\in
    \wlozp$ with $m(f)=1$, the absolute total charge of the corresponding solution to \eqref{P} is increasing
    and tends to infinity as $t\to \infty$. This means that there cannot be equilibria of \eqref{P} in $\wlozp$ except for the zero function. This is consistent with the results of this section. Indeed, we have $ \mathbf{I}_K=\emptyset$, since $\kappa=1$ and $\Lambda_\pm=1$. Hence {\bf$\mathcal{E}$} is not satisfied. Moreover, $\Tilde{\psi}(k)=1$ for all $k\in\Z$. This means that the nonnegative  functions that satisfy the detailed balance condition are simply exponential functions, which is also the case in the EDG model, cf.\ \cite{EDG}. But in contrast to the EDG model, the only function in $\lozp$ that satisfies the detailed balance condition \eqref{db} is $g\equiv 0$.
\end{exa}
%************************************************************
\section{Relative entropy} \label{EntropySec}
%*************************************************************
%***************************************************************
In this section, we always assume {\bf($\mathcal{B}$), ($\mathcal{P}$), ($\mathcal{DB}$)} and {\bf ($\mathcal{E}$)}.
We restrict ourselves to nonnegative functions with unit mass. This simplifies the presentation of the results and, in view of the scaling properties (and a suitable generalisation of relative entropy, see for example \cite{BGL}), entails no loss of generality. We set
%*************************************************************
\begin{align*}
    \po&\vcentcolon=\lkg f\in\lozp:m(f)=1\rkg, \quad
    \poo \vcentcolon=\po \cap \wlozp,\;\;\text{and}\\
    \pn&\vcentcolon=\lkg f\in \po: f(k)=0\;\text{for all}\;k\in \Z\setminus S_N\rkg,\quad N\in \N.
\end{align*}
%********************************************************
We begin by recalling the definition of relative entropy in the present setting.
%*********************************************
\begin{definition}\label{relentdef}
    Let $f,g\in\po$ and $g(k)>0$ for every $k\in\Z$. Then the \emph{relative entropy} of $f$ with respect to $g$ is defined by
     \begin{align*}
        H(f|g)\vcentcolon=\sum_{k\in\Z} f(k)\log\lk\frac{f(k)}{g(k)}\rk,
    \end{align*}
where, as usual, we adopt the convention $0\log(0)=0$.
\end{definition}
%***********************************************************************
\begin{lemma}\label{lement1}
Let $f,g\in\po$ and $g(k)>0$ for every $k\in\Z$. Then the following assertions hold.
\begin{enumerate}[(i)]
    \item The relative entropy is well-defined and takes values in $[0,\infty]$.
    \item $H(f|g)=0$ if and only if $f=g$.
    \item\label{lement1iii} If $f\in\poo$, $\Lambda_-,\Lambda_+<\infty$ and $\phi\in  \mathbf{I}_K$, then  $0\leq H(f|f_\phi)<\infty$.
\end{enumerate}
\end{lemma}
%*******************************************************************
\begin{proof}
        For (i) and (ii), we refer to Lemma 10.1 in \cite{Nag}. To prove (iii), note that
    \begin{align} \label{Hsplit}
        H(f|f_\phi)=&\sum_{k\in\Z}f(k)\log(f(k))-\sum_{k\in\Z}f(k)\log\lk \psi(k)\phi^k\rk+\sum_{k\in\Z}f(k)\log Z(\phi).
    \end{align}
The first term is finite by Lemma~4.2 in \cite{Ball}, while the third term is finite since $f\in\loz$. Concerning the second term, we write
    \begin{align*}
        \sum_{k\in\Z}f(k)\log\lk \psi(k)\phi^k\rk&=\sum_{k=1}^\infty k\left[ f(k)\log\lk \Tilde{\psi}(k)^{1/k}\kappa^{-1}\phi\rk\right.\\&\quad\quad+\left.f(-k)\log\lk \Tilde{\psi}(-k)^{1/k}\kappa^{-1}\phi^{-1}\rk\right]. 
    \end{align*}
    From the convergence properties of $\Tilde{\psi}$ (see \eqref{psitildeinfty}) and the assumptions on $\Lambda_-$, $\Lambda_+$ and $\phi$, it follows that $\log\lk \Tilde{\psi}(k)^{1/k}\kappa^{-1}\phi\rk_{k\in\N}$ and $\log\lk \Tilde{\psi}(-k)^{1/k}\kappa^{-1}\phi^{-1}\rk_{k\in\N}$ are bounded sequences. This, together with $f\in\wloz$, implies that the second term on the right of \eqref{Hsplit} converges in $\R$.
\end{proof}

We next look at continuity properties of relative entropy.

%***********************************************************************
\begin{lemma}\label{Hsemicontcont} 
(i) For every $g\in\po$ with $g(k)>0$ for every $k\in\Z$, the relative entropy $H(\,\cdot\,|g)$ is lower semicontinuous on $\po$ in the $\loz$-norm; in particular, this holds on $\poo$ in the $\wloz$-norm.

\smallskip
\noindent (ii) If $\Lambda_-,\Lambda_+<\infty$ and $\phi\in  \mathbf{I}_K$, then $H(\,\cdot\,|f_\phi)$ is continuous on $\poo$ with respect to the $\wloz$-norm. 
This also holds for $\phi=\phi_+$ ($\phi=\phi_-$) provided that $f_{\phi_+}$ ($f_{\phi_-}$) is well-defined and lies in $\poo$.
\end{lemma}
%********************************************************************
\begin{proof}
    (i) Since $\poo\subset\po$ and convergence in $\wloz$ entails convergence in $\loz$, it suffices to prove the assertion for $\po$.
    Let $(f_n)_{n\in\N}\subset\po$ and $f\in\po$ such that $f_n\rightarrow f$ in $\loz$ as $n\rightarrow\infty$. Setting
    $\Psi(z)=z\lke \log\lk z\rk-1\rke+1$, $z\ge 0$, and noting that $\Psi(z)\ge 0$ on $[0,\infty)$, we have
    \begin{align*}
        H(f|g)&=\sum_{k\in\Z} \Psi\Big(\frac{f(k)}{g(k)}\Big)g(k)=\sum_{k\in\Z}\liminf_{n\rightarrow\infty}\Psi\Big(\frac{f_n(k)}{g(k)}\Big) g(k)\\
        &\leq \liminf_{n\rightarrow\infty}\sum_{k\in\Z}
        \Psi\Big(\frac{f_n(k)}{g(k)}\Big) g(k)
        =\liminf_{n\rightarrow\infty}H(f_n|g),
    \end{align*}
    where we used Fatou's lemma.

(ii) We decompose the relative entropy into three terms as in \eqref{Hsplit}. The first term is continuous by \cite[Lemma 4.2]{Ball}, while the third is a scalar multiple of the $\loz$-norm and hence continuous. For the second term, we argue as in the last part of the proof of Lemma~\ref{lement1} \eqref{lement1iii}, which yields the desired continuity property.
\end{proof}
%********************************************************************
\begin{lemma}\label{lemfinent}
    Let $f\in\poo$. The following statements are equivalent.
    \begin{enumerate}[(i)]
        \item\label{lemfinenti} There exists $\phi\in \mathbf{I}_K$ with $Z(\phi)<\infty$ and $H(f|f_\phi)<\infty$.
        \item\label{lemfinentii} For all $\phi\in \overline{ \mathbf{I}_K}\cap(0,\infty)$ with $Z(\phi)<\infty$, one has $H(f|f_\phi)<\infty$.
        \item\label{lemfinentiii} The sum $\sum_{k\in\Z}f(k)\log\lk \psi(k)\rk$ converges in $\R$.
    \end{enumerate}
\end{lemma}
%********************************************************************
\begin{proof}
    \eqref{lemfinenti}$\Rightarrow$\eqref{lemfinentiii}:
    We decompose $H(f|f_\phi)$ into three terms as in \eqref{Hsplit}. The first and third terms are finite.
    If $H(f|f_\phi)<\infty$, then the second term must also converge in $\R$. Writing
    \[
    \sum_{k\in\Z}f(k)\log\lk \psi(k)\phi^k\rk=
    \sum_{k\in\Z}f(k)\log\lk \psi(k)\rk+
    \sum_{k\in\Z}k f(k)\log \phi,
    \]
    and noting that the last sum is finite, we obtain \eqref{lemfinentiii}.
    
\eqref{lemfinentiii}$\Rightarrow$\eqref{lemfinentii}: 
This follows by reversing the argument from the first part.

\eqref{lemfinentii}$\Rightarrow$\eqref{lemfinenti}: Since $ \mathbf{I}_K\neq \emptyset$ and $Z(\phi)<\infty$ for all $\phi\in  \mathbf{I}_K$ by {\bf ($\mathcal{E}$)}, \eqref{lemfinentii} implies \eqref{lemfinenti}.
\end{proof}

%**********************************************************************

\begin{proposition}\label{propweightedbound}
    Let $f,f_*\in \po$ with $f_*(k)>0$ for all $k\in \Z$. Let $h:\Z\rightarrow[0,\infty)$ be such that 
    \[\sum_{k\in \Z}h(k)\exp\Big(-\frac{h(k)}{2}\Big)<\infty.\]
    Assume that there exists $N_0\in \N$ such that 
    \begin{align} \label{fstarbound}
    f_*(k)\leq \exp(-h(k))\quad\text{for all}\;\,k\in\Z\;\,\text{with}\;\,|k|\geq N_0.
    \end{align}
    Then
    \begin{align} \label{weightedsum}
       \skz h(k)f(k)\leq \max_{|k|<N_0}h(k)+2H(f|f_*)+
       \sum_{k\in \Z}h(k)\exp\Big(-\frac{h(k)}{2}\Big)+2.
    \end{align}
\end{proposition}

%***********************************************************************

\begin{proof}
    Since $f\in \po$, we have 
    \begin{align} \label{ZerlegungN0}
        \skz h(k)f(k)\leq \max_{|k|<N_0}h(k)+\sum_{k\geq N_0} h(k)f(k)+
        \sum_{k\leq -N_0}h(k)f(k).
    \end{align}
   To estimate the second summand on the right-hand side, we define
    \begin{align*}
        M_1&\vcentcolon=\lkg k\geq N_0:\frac{f(k)}{f_*(k)}\geq e^\frac{h(k)}{2}\rkg \text{ and } 
        M_2\vcentcolon=\lkg k\geq N_0:\frac{f(k)}{f_*(k)}<e^\frac{h(k)}{2}\rkg.
    \end{align*}
    Then $h(k)/2\leq\log\lk\frac{f(k)}{f_*(k)}\rk$ for all $k \in M_1$, and $f(k)<\exp(-h(k)/2)$ for all $k\in M_2$, where we use \eqref{fstarbound}. Hence
    \begin{align*}
        \sum_{k\geq N_0}h(k)f(k)\leq2\sum_{k\in M_1}f(k)\log\lk\frac{f(k)}{f_*(k)}\rk+\sum_{k\in M_2}
        h(k)\exp\Big(-\frac{h(k)}{2}\Big).
    \end{align*}
    Moreover,
    \begin{align*}
         \sum_{k\in M_1}f(k)\log\lk\frac{f(k)}{f_*(k)}\rk &\leq \sum_{k\in M_1}\lke\frac{f(k)}{f_*(k)}\lk \log\lk\frac{f(k)}{f_*(k)}\rk-1\rk+1\rke f_*(k)\\
        &\quad\quad+\sum_{k\in M_1}f(k)-\sum_{k\in M_1}f_*(k).
    \end{align*}
    
    The third term on the right-hand side of \eqref{ZerlegungN0} is estimated analogously. Combining all estimates yields \eqref{weightedsum}.
\end{proof}
%*********************************************************************
We now turn to monotonicity of relative entropy along solutions of \eqref{P} with respect to detailed balance equilibria. 
%********************************************************************
\begin{definition}
For $k,l\in\Z$ and $g\in\loz$, we define
    \begin{align*}
        j_{l,k-1}[g] &\vcentcolon=K(l,k-1)g(l)g(k-1),\\
        J_{k-1}[g]& \vcentcolon=\sum_{l\in\Z}\lk j_{l,k-1}[g]-j_{k,l-1}[g]\rk.
    \end{align*}
    Supposing in addition that $g$ takes values in $(0,\infty)$, we further define
    \begin{align*}
        &W(g)\vcentcolon=\frac{1}{2} \sum_{k,l\in\Z} \big( j_{l+1,k}[g]-j_{k+1,l}[g]\big)\lk\log\lk \frac{j_{l+1,k}[g]}{j_{k+1,l}[g]}\rk\rk,\\
        W^N(g)&\vcentcolon=\frac{1}{2} \sum_{k,l=-N}^{N-1} \lk j_{l+1,k}[g]-j_{k+1,l}[g]\rk\lk\log\lk \frac{j_{l+1,k}[g]}{j_{k+1,l}[g]}\rk\rk,\quad N\in \N.
    \end{align*}
\end{definition}
%****************************************************************************
\begin{re}\label{Wre}
For $g\in\loz$ with $g(k)>0$ for all $k\in \Z$, the following statements hold.
\begin{itemize}
\item [(i)] $W(g)\in [0,\infty]$ and $W^N(g)\in [0,\infty)$ for all $N\in\N$, since $(x-y)\log(x/y)\ge 0$ for all $x,y>0$. 
\item [(ii)] $W(g)\leq \liminf_{N\rightarrow\infty}W^N(g)$, by Fatou's lemma.
\item [(iii)] $g$ satisfies \eqref{db} if and only if $W(g)=0$, since for all $x,y>0$,
    \begin{align*}
        (x-y)\log(x/y)=0 \Leftrightarrow x=y.
    \end{align*}
    \end{itemize}
\end{re}
%********************************************************************************

We recall that, by Remark \ref{DBRemark} (ii), 
for every nonnegative equilibrium $f_*\neq 0$ of \eqref{P} satisfying \eqref{db}, one has $f_*(k)>0$ for all $k\in \Z$.
%******************************************************
\begin{proposition}\label{ddtHN}
Let $N\in \N$, $f_0\in\pn$, and let $f_N$ be the solution to \eqref{IVPN} with initial value $f_0\,(=f_0^N)$. Furthermore, let $f_*$ be a nonnegative equilibrium of \eqref{ODE} with unit mass satisfying \eqref{db}. Then
\[
\frac{d}{dt}H(f_N(t)|f_*)=-W^N(f_N(t)),\quad t>0.
\]
\end{proposition}
%*******************************************************
\begin{proof}
Note first that $f_N(t,k)>0$ for all $t>0$ and $k\in S_N$. This follows from the fact that $K^N(k,l)>0$ for all $k\in\{-N+1,\dots,N\}$ and $l\in\{-N,\dots,N-1\}$ and by arguing as in the proof of Proposition \ref{propfg0}.

The following computation is analogous to that in the proof of Proposition 3.4 in \cite{EDG}.
    \begin{align*} 
        \frac{d}{dt}H(&f_N(t)|f_*)=\sum_{k=-N}^N\frac{d}{dt}\lk f_N(t,k)\log\lk\frac{f_N(t,k)}{f_*(k)}\rk\rk \\ &= \sum_{k=-N}^NQ^N(f_N(t,\cdot))(k)\log\lk\frac{f_N(t,k)}{f_*(k)}\rk \\ 
        &= \sum_{k=-N}^{N-1}\lk J_{k-1}[f_N(t,\cdot)]-J_k[f_N(t,\cdot)]\rk\log\lk\frac{f_N(t,k)}{f_*(k)}\rk \\
        &= \sum_{k=-N}^{N-1}J_k[f_N(t,\cdot)]\lk\log\lk\frac{f_N(t,k+1)}{f_*(k+1)}\rk-\log\lk\frac{f_N(t,k)}{f_*(k)}\rk\rk  \\
        &= \sum_{k,l=-N}^{N-1} \lk j_{l+1,k}[f_N(t,\cdot)]-j_{k+1,l}[f_N(t,\cdot)]\rk\\
        &\quad\quad\quad\quad\quad\quad\cdot\lk\log\lk\frac{f_N(t,k+1)}{f_*(k+1)}\rk-\log\lk\frac{f_N(t,k)}{f_*(k)}\rk\rk\\  
        &= \frac{1}{2}\sum_{k,l=-N}^{N-1} \lk j_{l+1,k}[f_N(t,\cdot)]-j_{k+1,l}[f_N(t,\cdot)]\rk\\
        &\quad\quad\quad\quad\quad\quad\cdot\lk\log\lk\frac{f_N(t,k+1)}{f_*(k+1)}\rk-\log\lk\frac{f_N(t,k)}{f_*(k)}\rk\rk
        \end{align*}
        \begin{align*}
        \hphantom{\frac{d}{dt}H(f|f_*)=}&\quad+\frac{1}{2}\sum_{k,l=-N}^{N-1} \lk j_{k+1,l}[f_N(t,\cdot)]-j_{l+1,k}[f_N(t,\cdot)]\rk\\
        &\quad\quad\quad\quad\quad\quad\cdot\lk\log\lk\frac{f_N(t,l+1)}{f_*(l+1)}\rk-\log\lk\frac{f_N(t,l)}{f_*(l)}\rk\rk  \\
        &=\frac{1}{2} \sum_{k,l=-N}^{N-1} \lk j_{l+1,k}[f_N(t,\cdot)]-j_{k+1,l}[f_N(t,\cdot)]\rk\left(\log\lk\frac{f_N(t,k+1)}{f_*(k+1)}\rk\right.\\
        &\quad\quad\quad\quad\quad\quad\left.-\log\lk\frac{f_N(t,k)}{f_*(k)}\rk-\log\lk\frac{f_N(t,l+1)}{f_*(l+1)}\rk+\log\lk\frac{f_N(t,l)}{f_*(l)}\rk\right)  \\
        &=\frac{1}{2} \sum_{k,l=-N}^{N-1} \lk j_{l+1,k}[f_N(t,\cdot)]-j_{k+1,l}[f_N(t,\cdot)]\rk\\
        &\quad\quad\quad\quad\quad\quad\cdot\lk\log\lk \frac{f_N(t,k+1)f_N(t,l)}{f_*(k+1)f_*(l)}\cdot \frac{f_*(k)f_*(l+1)}{f_N(t,k)f_N(t,l+1)} \rk\rk\\
        \hphantom{\frac{d}{dt}H(f|f_*)}&=\frac{1}{2} \sum_{k,l=-N}^{N-1}\lk j_{l+1,k}[f_N(t,\cdot)]-j_{k+1,l}[f_N(t,\cdot)]\rk\\
        &\quad\quad\quad\quad\quad\quad\cdot\left(\log\left( \frac{K(k+1,l)f_N(t,k+1)f_N(t,l)}{K(k+1,l)f_*(k+1)f_*(l)}\right.\right.\\
        &\quad\quad\quad\quad\quad\quad\quad\quad\quad\quad\quad\quad\left.\left.\cdot \frac{K(l+1,k)f_*(k)f_*(l+1)}{K(l+1,k)f_N(t,k)f_N(t,l+1)} \right)\right)\\
        &=\frac{1}{2} \sum_{k,l=-N}^{N-1} \lk j_{l+1,k}[f_N(t,\cdot)]-j_{k+1,l}[f_N(t,\cdot)]\rk\lk\log\lk \frac{j_{k+1,l}[f_N(t,\cdot)]}{j_{l+1,k}[f_N(t,\cdot)]}\rk\rk\\ 
        &=-W^N(f_N(t,\cdot)),
    \end{align*}
    where in the penultimate step we used that $f_*$ satisfies \eqref{db}. 
\end{proof}
%*********************************************************************************

\begin{cor}\label{HWNint} Let $N\in \N$, $f_0\in\pn$, and let $f_N$ be the solution to \eqref{IVPN} with initial value $f_0\,(=f_0^N)$. Moreover, let $f_*$ be a nonnegative equilibrium of \eqref{ODE} with unit mass satisfying \eqref{db}. Then, for all $0\leq s\leq t$,
    \begin{align*}
        H(f_N(t)|f_*)+\int_s^t W^N(f_N(\tau))\,d\tau=H(f_N(s)|f_*).
    \end{align*}
\end{cor}
%************************************************************************************
\begin{proof}
    For $0<s\leq t$, this is trivial. For $s=0$ this can be obtained by  monotone convergence (since $W^N(f^N(\tau))$ is nonnegative and well-defined for all $\tau>0$; the value for $\tau=0$ is irrelevant) and the fact that $H(f_N(\,\cdot\,)|f_*)$ is continuous.
\end{proof}
%**************************************************************************
\begin{cor}\label{corHWleqH}
    Let $f_0\in\poo$, $f$ the corresponding solution to \eqref{P}, and let $f_*\in \po$ be a detailed balance equilibrium of \eqref{ODE}. Then, for all $0\leq s\leq t$,
    \begin{align}\label{hwlh}
        H(f(t)|f_*)+\int_s^t W(f(\tau))\,d\tau\leq H(f(s)|f_*).
    \end{align}
    Here, the right-hand side or both sides can take the value $\infty$. In particular, if $H(f_0|f_*)<\infty$, then $t\mapsto H(f(t)|f_*)\in [0,\infty)$ is nonincreasing.
\end{cor}
%**********************************************************************
\begin{proof}
    The argument is inspired by the proof of \cite[Theorem 4.8]{Ball}. 
    
    Let $s\geq 0$ and let $f_N$ be the solution to \eqref{IVPN} with initial value $\frac{f(s)^N}{m(f(s)^N)}$ for all $N\in\N$ with $m(f(s)^N)>0$ (which is always satisfied if $s>0$, due to Proposition~\ref{propfg0}; and for $s=0$, this is true for all sufficiently large $N$, since $m(f_0)=1$). Note that $m(f(s)^N)\to 1$ and thus $H(f_N(0)|f_*)\rightarrow H(f(s)|f_*)$ as $N\to \infty$. Moreover, Lemma~\ref{Hsemicontcont} and Remark~\ref{Wre} imply
    $H(f(t)|f_*)\leq\liminf_{N\rightarrow\infty}H(f_N(t-s)|f_*)$ and $W(f(\tau))\leq\liminf_{N\rightarrow\infty}W^N(f_N(\tau-s))$ for all $s\leq\tau\leq t$ with $\tau\neq0$. Hence, by Corollary~\ref{HWNint},
    \begin{align*}
        H(f(t)|f_*)\,+&\int_s^tW(f(\tau))\,d\tau\\ &\leq\liminf_{N\rightarrow\infty}H(f_N(t-s)|f_*)+\int_s^t \liminf_{N\rightarrow\infty}W^N(f_N(\tau-s))\,d\tau\\ &\leq \liminf_{N\rightarrow\infty}\lk H(f_N(t-s)|f_*)+\int_s^t W^N(f_N(\tau-s))\,d\tau\rk\\&= \liminf_{N\rightarrow\infty} H(f_N(0)|f_*)=H(f(s)|f_*).\qedhere
    \end{align*}
\end{proof}
%*************************************************************************
\begin{cor} \label{l11bound}
    Let $f_0\in\poo$ and $f$ be the corresponding solution to \eqref{P}. Then the trajectory $\{f(t):t\geq0\}$ of $f$ is bounded in $\wloz$.
\end{cor}
%********************************************************************
\begin{proof}
    By {\bf($\mathcal{E}$)}, for every $\phi\in\mathbf{I}_K$ there exist $h:\Z\rightarrow[0,\infty)$, $N_0\in\N$  and $\delta>0$ such that $f_\phi(k)\leq \exp(-h(k))$ and $h(k)\geq \delta|k|$ for all $k\in\Z$ with $|k|\geq N_0$. Thus, the assumptions of Proposition~\ref{propweightedbound} are met for all such $\phi$. Now, fix $\phi\in  \mathbf{I}_K$. Since the relative entropy $H(f(t)|f_\phi)$ is nonincreasing, we obtain that $\sup_{t\geq0}\lno f(t)\rnooo$ is bounded from above by
    \begin{align*}
         \frac{1}{\delta}\lk\max_{|k|<N_0}h(k)+2H(f_0|f_\phi)+\sum_{|k|\geq N_0}h(k)\exp\Big(-\frac{h(k)}{2}\Big)+2\rk+1<\infty,
    \end{align*}
    where we also use mass conservation.
\end{proof}
%********************************************************************
\begin{re}
    If $f_*\in\poo$ is an equilibrium of \eqref{ODE}, then $f_*$ satisfies the detailed balance condition \eqref{db}. In fact, from Remark \ref{reequipos} we already know that $f_*(k)>0$ for all $k\in \Z$. Next,
    note that for all $0\leq s\leq t$,
    \begin{align*}
        0\leq H(f_*|f_*)+\int_s^tW(f_*)\,d\tau\leq H(f_*|f_*)=0,
    \end{align*}
    and thus $W(f_*)=0$, which, by Remark~\ref{Wre}, implies that $f_*$ satisfies \eqref{db}. 
    
    In particular, this means that in case of Example~\ref{ExaFinMaxCharge}, we cannot expect convergence in $\wloz$ for all initial values in $\poo$. The reason for this is that convergence in $\wloz$ preserves total charge and for sufficiently large total charge, there is no equilibrium with this charge.
\end{re}
%*************************************************************************

%****************************************
\section{Stability}
%**************************************
The following theorem provides a useful weighted Csiszár-Kullback-Pinsker (CKP) inequality in a general setting. It constitutes a key ingredient in the proof of our stability result, Theorem \ref{StabTheorem}.
%*********************************************************************
\begin{theorem} \label{bvt}(\cite[Theorem 2.1]{BV}) 
    Let $\Omega$ be a measurable space, let $\mu,\nu$ be two probability measures on $\Omega$ and let $\varphi$  be a nonnegative measurable function on $\Omega$. Then
    \begin{align}
        \lVert \varphi\mu-\varphi\nu\rVert_{TV}\leq \lk \frac{3}{2}+\int_\Omega e^{2\varphi(x)}\,d\nu(x)\rk\lk\sqrt{H(\mu|\nu)}+\frac{1}{2}H(\mu|\nu)\rk.
    \end{align}
\end{theorem}
%**************************************************************************
\begin{re}
    Theorem~\ref{bvt} can also be used to obtain boundedness in $\wloz$ of solutions to \eqref{P} with $f_0\in \poo$, cf.\ Corollary \ref{l11bound}. In this context, we mention that a priori bounds via suitable weighted CKP inequalities were established in \cite[Section 2]{Nh02} for the Becker-Döring equations. 
\end{re}
%***************************************************************************
\begin{theorem} \label{StabTheorem}
Assume {\bf($\mathcal{B}$), ($\mathcal{P}$), ($\mathcal{DB}$), ($\mathcal{E}$)} and $\Lambda_\pm<\infty$. Then for all $\phi\in  \mathbf{I}_K$, the equilibrium $f_\phi$ is stable in $\poo$, that is, for every $\varepsilon>0$ there exists $\delta>0$ such that for all $f_0\in \poo$ satisfying $\lno f_0-f_*\rnooo<\delta$, we have $\lno f(t)-f_*\rnooo<\varepsilon$ for all $t\geq0$, where $f$ is the solution to \eqref{P} with initial value $f_0$.
\end{theorem}
%************************************************************************
\begin{proof}
    Let $\phi\in  \mathbf{I}_K$. Choose $\alpha>0$ such that $\skz e^{2\alpha(1+|k|)}f_\phi(k)<\infty$, which is possible according to Remark~\ref{reequiwell}. Define $\varphi:\Z\rightarrow (0,\infty)$ by $\varphi(k)= \alpha(1+|k|)$.

    Let $\varepsilon>0$. By the continuity of $H(\,\cdot\,|f_\phi)$ on $\poo$ (Lemma~\ref{Hsemicontcont}) and $H(f_\phi|f_\phi)=0$, it follows that for every $\delta_2>0$ there exists $\delta_1>0$ such that for all $f_0\in\poo$ with $\lno f_0-f_\phi\rnooo<\delta_1$ we also have $H(f_0|f_\phi)<\delta_2$. 
    Applying Theorem~\ref{bvt} yields
    \begin{align*}
        \alpha &\lno f(t)-f_\phi\rnooo=\lno\varphi f(t) -\varphi f_\phi\rno_{TV}\\ &\leq \lk \frac{3}{2}+\skz e^{2\alpha(1+|k|)}f_\phi(k)\rk\lk\sqrt{H(f(t)|f_\phi)}+\frac{1}{2}H(f(t)|f_\phi)\rk,\quad t\ge 0.
    \end{align*}
    As the relative entropy is nonincreasing along solutions to \eqref{P} (Corollary~\ref{corHWleqH}), we infer that 
    \begin{align}\label{stabeq}
        \alpha\lno f(t)-f_\phi\rnoo<\lk \frac{3}{2}+\skz e^{2\alpha(1+|k|)}f_\phi(k)\rk\frac{\sqrt{\delta_2}+\frac{1}{2}\delta_2}{\alpha},\quad t\ge 0.
    \end{align}
    Hence, choosing $\delta_1$ so small that the right-hand side of \eqref{stabeq} is smaller than $\varepsilon$, yields the assertion.
\end{proof}
%****************************************************************
\begin{re}
If $f_{\phi_+}$ ($f_{\phi_-}$) is well-defined and lies in $\poo$, then $f_{\phi_+}$ ($f_{\phi_-}$) is stable in $\poo$ by the same argument, since the corresponding relative entropy is also continuous on $\poo$, thanks to Lemma~\ref{Hsemicontcont}.
\end{re}

\appendix
\section{Appendix}
%**********************************************************
\subsection{On the absolute total charge for the discrete heat equation}
%**********************************************************
The solution $f$ of the discrete heat equation on $\iZ$, equation \eqref{DiscreteHeat}, with initial value $f(0,\cdot)=f_0\in \wloz$ can be represented by
\[
f(t,k)=\sum_{l\in\iZ}G(t,k-l)f_0(l),\quad t\ge 0,\,k\in \iZ,
\]
where
\begin{equation} \label{Gformula}
G(t,k)=\frac{e^{-2t}}{2\pi}\int_{-\pi}^\pi e^{2t\cos \theta}e^{-ik\theta}\,d\theta,\quad t\ge 0,\,k\in \iZ, 
\end{equation}
see \cite{AGMP} and the proof of \cite[Prop.\ 1]{CGRTV}. Assuming that $f_0$ is nonnegative, we have $f(t,0)\ge f_0(0)G(t,0)$, and thus, in view of \eqref{DerivAbsCharge},
\begin{align*}
    |q|(f(t,\cdot))\ge |q|(f_0)+2f_0(0)\int_0^t G(\tau,0)\,d\tau,\quad t\ge 0.
\end{align*}
Using \eqref{Gformula} and Tonelli's theorem, we further have
\begin{align*}
\int_0^t G(\tau,0)\,d\tau= \frac{1}{2\pi}\int_{-\pi}^{\pi}
\int_0^t e^{2\tau(\cos \theta-1)}\,d\tau\,d\theta
=\frac{1}{4\pi}\int_{-\pi}^{\pi}\frac{1-e^{-2t(1-\cos \theta)}}{1-\cos \theta} \,d\theta\to \infty
\end{align*}
as $t\to \infty$. This shows that $|q|(f(t,\cdot))\to \infty$
as $t\to \infty$ whenever $f_0(0)>0$. Since $f$ becomes
instantaneously positive on $\iZ$ for any nonnegative $f_0\in \wloz\setminus\{0\}$, we see that for any such initial value, the corresponding absolute total charge tends to infinity as $t\to \infty$.   
%******************************************************
\subsection{Preservation of positivity due to quasi-positivity}
\begin{proposition}\label{quasipositivitypositivitypreserving}
    Let $X\in\{\loz,\wloz\}$, 
    $T>0$, and let $F:X\rightarrow X$ be locally Lipschitz continuous, i.e.
    \begin{gather} \label{LocLip}
    \begin{aligned}
        &\forall y_0\in X\; \exists r,L>0:\quad\lno F(x)-F(y)\rno_X\leq L\lno x-y\rno_X\;\; \forall x,y\in B_r(y_0).
    \end{aligned}
    \end{gather}
    Moreover, assume that $F$ is quasi-positive, that is,
    \begin{gather*}
        \begin{aligned}
            \forall k\in \Z\;\forall y\in X^+:\quad \lke y(k)=0\Rightarrow F(y)(k)\geq 0\;\; \forall t\in [0,T)\rke.
        \end{aligned} 
    \end{gather*}
    If $f_0\in X^+$ and $f\in C^1([0,T);X)$ satisfies 
    $\dot{f}(t)=F(f(t))$ on $[0,T)$ and $f(0)=f_0$, then $f(t)\in X^+$ for all $t\in[0,T)$.
\end{proposition}

For the proof of Proposition~\ref{quasipositivitypositivitypreserving} we use Theorem 1 in \cite{Volkmann}. The relevant part is the following.
\begin{theorem}\label{ThmV}
    Let $C$ be a nonempty, closed and convex subset of a Banach space $E$, and let $F:\,E\rightarrow E$ be locally Lipschitz continuous (cf.\ \eqref{LocLip}).
    Furthermore, assume that $F$ satisfies the following condition.
        \begin{equation*}
            \varphi\in E^*,\varphi\neq0,\,x\in C\text{ and }\varphi(x)=\inf_{y\in C}\varphi(y)\;\Rightarrow\;\varphi(F(x))\geq0.
        \end{equation*}
    Let $0<T\leq\infty$ and $f\in C^1([0,T);E)$ be such that $f(0)\in C$ and $\dot{f}(t)=F(f(t))$ for all $0<t<T$. Then 
    \begin{align*}
        f(t)\in C \,\text{ for all }\; t\in[0,T).
    \end{align*}
\end{theorem}

\begin{proof}[Proof of Proposition~\ref{quasipositivitypositivitypreserving}] Let $C=X^+$.
    According  to Theorem~\ref{ThmV} it suffices to show that whenever 
    \begin{align}\label{Volkmannbedingung}
        \varphi\in X^*,\,\varphi\not\equiv0,\,x\in C,\,\varphi(x)=\inf_{y\in C}\varphi(y),
    \end{align}
    then we also have
    \begin{align}\label{volkmannvarfgeq0}
        \varphi(F(x))\geq0.
    \end{align}
    Assume \eqref{Volkmannbedingung} holds. We claim that
    \begin{enumerate}[(a)]
        \item\label{apositivityconverved} $\varphi(x)=0$.
        \item\label{bpositivityconverved} $\varphi(y)\geq0$ for all $y\in C$.
    \end{enumerate}
    Indeed, since $0\in C$ and $\varphi\in X^*$, we clearly have $\varphi(x)\leq\varphi(0)=0$. Now suppose $\varphi(x)<0$. Then it follows that $\varphi(2x)=\varphi(x)+\varphi(x)<\varphi(x)$, which is a contradiction as $2x\in C$. This shows \eqref{apositivityconverved}.
    \eqref{bpositivityconverved} follows directly from \eqref{apositivityconverved}.
    
    Next, we prove that
    \begin{align*}
        (F(x))(k) \varphi(e_k)\geq0,\quad k\in \Z,
    \end{align*}
    where $e_k$ is the $k$-th unit vector, i.e. $e_k(l)=\delta_{kl}$ (Kronecker symbol).
    To this end, let $k\in \Z$. We distinguish two cases.
    If $\varphi(e_k)=0$, then it clearly follows that $F(x)(k) \varphi(e_k)=0$. Suppose now that
    $\varphi(e_k)\neq0$. $x\in C$ implies $x(k)\geq0$. Suppose $x(k)>0$. 
    Then, together with $x-x(k)e_k\in C$, this yields 
    \begin{align*}
        \varphi\big(x-x(k)e_k\big)=\varphi(x)-x(k)\varphi(e_k)<\varphi(x),
    \end{align*}
    a contradiction. Therefore, we have $x(k) = 0$ and thus, by the quasi-positivity of $F$, $F(x)(k)\geq0$. From $e_k\in C$ and \eqref{bpositivityconverved} we infer $F(x)(k)\varphi(e_k)\geq0$.
    
    By continuity of $\varphi$, we conclude
    \begin{align*}
        \varphi( F(x))=\sum_{k\in \Z}F(x)(k)\varphi(e_k)\geq0,
    \end{align*}
    which is exactly \eqref{volkmannvarfgeq0}.
\end{proof}

\subsection{Precompactness in sequence spaces}

The following proposition provides a compactness criterion for subsets of $\wloz$, which is known as Lemma of de La Vallée-Poussin. A corresponding result including proof for $\ell_{1,1}(\N_0)$ can be found in
\cite[Lemma A.1]{EDG}. The proof given there can be easily adapted to $\wloz$. 
%****************************************************************
\begin{proposition}\label{valleepoussin}
    Let $S\subset\wloz$. Then the following statements are equivalent.
    \begin{enumerate}[(i)]
        \item\label{vp1} $S$ is a precompact subset of $(\wloz,\lno\cdot\rno_{1,1})$.
        \item\label{vp2} $S$ is bounded in $(\wloz,\lno\cdot\rno_{1,1})$ and  \begin{align*}
            \lim_{N\rightarrow\infty}\sup_{f\in S}\sum_{|k|>N}(1+|k|)|f(k)|=0.
        \end{align*}
        \item\label{vp3} There exists a positive sequence $x:\N_0\rightarrow(0,\infty)$ such that
        \begin{itemize}
            \item $\frac{x(k)}{(1+k)}$ is nondecreasing,
            \item $\frac{x(k)}{(1+k)}\rightarrow\infty$ as $k\rightarrow\infty$, and
            \item $\displaystyle\sup_{f\in S}\sum_{k\in\Z}x(|k|)|f(k)|<\infty.$
        \end{itemize}
    \end{enumerate}
\end{proposition}
%**********************************************************
\noindent{\bf Acknowledgement:} The authors are grateful to Andr\'{e} Schlichting for several fruitful and stimulating discussions that contributed to the development of this work.
%**********************************************************
%*******************************************************************

\end{document}